\documentclass[11pt]{amsart}
\usepackage{amsmath,amsthm,amscd,amsfonts, amssymb,mathrsfs}
\usepackage{graphicx}
\usepackage{tikz}
\usepackage{color}
\usepackage{shuffle}
\include{diagxy}
\usepgflibrary{shapes}
\newcommand{\sect}[1]{\section{#1}\setcounter{equation}{0}}

\font\mbn=msbm10 scaled \magstep1
\font\mbs=msbm7 scaled \magstep1
\font\mbss=msbm5 scaled \magstep1
\newfam\mbff
\textfont\mbff=\mbn
\scriptfont\mbff=\mbs
\scriptscriptfont\mbff=\mbss

\newcommand{\N}       { \mathbb{N}}
\newcommand{\Z}        {\mathbb{Z}  }

\newtheorem{Th}{Theorem}[section]

\newtheorem{C}[Th]{Corollary}

\newtheorem{Proposition}[Th]{Proposition}
\newtheorem{R}[Th]{Remark}
\newtheorem{E}[Th]{Example}

\newtheorem*{P3'}{Problem 3$'$}

\begin{document}

\title[Shuffle and Fa\`{a} di Bruno Hopf Algebras in the Center Problem for ODEs]{Shuffle and Fa\`{a} di Bruno Hopf Algebras in the Center Problem for Ordinary Differential Equations}

\author{Alexander Brudnyi} 
\address{Department of Mathematics and Statistics\newline
\hspace*{1em} University of Calgary\newline
\hspace*{1em} Calgary, Alberta\newline
\hspace*{1em} T2N 1N4}
\email{albru@math.ucalgary.ca}

\thanks{Research supported in part by NSERC}

\begin{abstract}
In this paper we describe the Hopf algebra approach to the center problem for the differential equation $\frac{dv}{dx}=\sum_{i=1}^{\infty}a_{i}(x)v^{i+1}$, $x\in [0,T]$, and study some combinatorial properties of the first return map of this equation. The paper summarizes and extends previously developed approaches to the center problem due to Devlin and the author.
\end{abstract}

\date{} 

\maketitle

\tableofcontents

\sect{Introduction}

Given the ordinary differential equation
\begin{equation}\label{e1}
\frac{dv}{dx}=\sum_{i=1}^{\infty}a_{i}(x)v^{i+1},\ \ \ x\in I_{T}:=[0,T],
\end{equation}
with coefficients $a_{i}\in L^{\infty}(I_{T})$ (- the Banach space of bounded measurable complex-valued functions on $I_{T}$ equipped with supremum norm) satisfying
\begin{equation}\label{e2}
\sup_{x\in I_{T},\ i\in\N}\sqrt[i]{|a_{i}(x)|}<\infty
\end{equation}
the {\em center problem} asks whether \eqref{e1} determines a {\em center}, i.e. whether every solution $v$ of \eqref{e1} with a sufficiently small initial value (existing and Lipschitz due to \eqref{e2}) satisfies $v(T)=v(0)$.
The center problem arises naturally in the framework of the geometric theory of ordinary differential equations created by Poincar\'{e}. In particular, it is related to the classical Poincar\'{e} Center-Focus problem for planar polynomial vector fields 
\begin{equation}\label{e3}
\frac{dx}{dt}=-y+F(x,y),\ \ \ \frac{dy}{dt}=x+G(x,y),
\end{equation}
where $F$ and $G$ are real polynomials of a given degree without constant and linear terms, asking about conditions on $F$ and $G$ under which all trajectories of (\ref{e3}) situated in a small neighbourhood of $0\in\mathbb R^{2}$ are closed. Passing to polar coordinates $(x,y)=(r\cos\varphi, r\sin\varphi)$ in (\ref{e3}) and expanding the right-hand side of the resulting equation as a series in $r$ (for $F$, $G$ with sufficiently small coefficients) one obtains an equation (\ref{e1}) with coefficients being trigonometric polynomials depending polynomially on the coefficients of (\ref{e3}). This transforms the Center-Focus problem to the center problem for equations (\ref{e1}) with coefficients depending polynomially on a parameter. (For recent advances in the area of the center problem for equation \eqref{e1} see  \cite{AL, A1, A2, A3, A4, Br6, Br2, Br3, Br4, Br5, BFY, BRY, Cr, CGM1, CGM2, CGM3, D, GGL, GGS, P1, P2}
and references therein.)\smallskip

By $\mathscr X$ we denote the vector space of sequences $a=(a_{1},a_{2},\dots)$ of coefficients of equation \eqref{e1} satisfying (\ref{e2}) and by
$\mathscr C\subset \mathscr X$ the set of centers of (\ref{e1}). Let  $v(x;r;a)$, $x\in I_{T}$, be the Lipschitz solution with initial value $v(0;r;a)=r$ of equation (\ref{e1}) with the sequence of coefficients $a\in \mathscr X$. Clearly, for every $x\in I_{T}$,  $v(x;r;a)\in G_{c}[[r]]$, the set of locally convergent near zero power series of the form $r+\sum_{i=1}^\infty c_i r^{i+1}$ with all $c_i\in\mathbb C$.
By definition, $\mathcal P(a):=v(T;\cdot;a)$ is the {\em first return map} of \eqref{e1}. The explicit expression for $\mathcal P(a)$ was obtained by Devlin \cite{D} (for equations with finitely many nonzero coefficients $a_i$) and  independently and by a different method by the author \cite{Br1} (the general case):
\begin{equation}\label{e4}
\mathcal P(a):=r+\sum_{i=1}^{\infty}\left(\sum_{i_{1}+\cdots +i_{k}=i}p_{i_{1},\dots, i_{k}}(i)\cdot I_{i_{1},\dots, i_{k}}(a)\right)r^{i+1}
\end{equation}
(in the inner sum $k$ runs over the set of natural numbers $1,\dots, i$),
where for $t\in\mathbb C$,
\begin{equation}\label{e5}
\begin{array}{l}
\displaystyle
p_{i_{1},\dots, i_{k}}(t)=
(t-i_{1}+1)(t-i_{1}-i_{2}+1)(t-i_{1}-i_{2}-i_{3}+1)\cdots (t-i+1)\\
\\
\displaystyle
{\rm and}\quad I_{i_{1},\dots, i_{k}}(a):=\int\cdots\int_{0\leq s_{1}\leq\cdots\leq s_{k}\leq T}a_{i_{k}}(s_{k})\cdots a_{i_{1}}(s_{1})\ \!ds_{k}\cdots ds_{1}.
\end{array}
\end{equation}
In particular, one obtains (see \cite[Th.3.1]{Br2})
\begin{equation}\label{e6}
\begin{array}{l}
\displaystyle
a\in {\mathscr C}\ \Longleftrightarrow\ \sum_{i_{1}+\cdots +i_{k}=i}p_{i_{1},\dots, i_{k}}\cdot I_{i_{1},\dots, i_{k}}(a)\equiv 0\quad {\rm for\ all}\quad i\in\N\medskip\\
\displaystyle \qquad\quad \Longleftrightarrow\ \sum_{i_{1}+\cdots +i_{k}=i}p_{i_{1},\dots, i_{k}}(i)\cdot I_{i_{1},\dots, i_{k}}(a)= 0
\quad {\rm for\ all}\quad  i\in\N.
\end{array}
\end{equation}

In \cite{D} the center problem for equation \eqref{e1} with finitely many nonzero coefficients $a_i$ was reformulated using the language of word-problems. In the same vein, in \cite{Br2} the algebraic model  for the center problem for equation \eqref{e1} was constructed. In the present paper we continue this line of research and describe the Hopf algebra approach to the center problem. The key point of this approach is that the first return map of equation \eqref{e1} determines the natural monomorphism of  the co-opposite of the {\em Fa\`{a} di Bruno Hopf algebra} into the {\em shuffle Hopf algebra} (see Sections 2--5 below for the corresponding definitions and results) given in terms of polynomials
\begin{equation}\label{e7}
\mathcal P_i(X_1,\dots, X_i)=\sum_{i_{1}+\cdots +i_{k}=i}p_{i_{1},\dots, i_{k}}(i)\, X_{i_{k}}\cdots X_{i_{1}},\quad i\in\N,
\end{equation}
in free noncommutative variables $X_1,X_2,\dots$.

Some recurrence relations for such polynomials (in the sequel called the {\em displacement polynomials} as they originated from the expression for the displacement map of equation \eqref{e1}) were established by Devlin, see \cite[Th.6.5]{D}. In the present paper we study
more general polynomials 
\begin{equation}\label{e8}
\widetilde{\mathcal P}_i(X_1,\dots, X_i; t)=\sum_{i_{1}+\cdots +i_{k}=i}p_{i_{1},\dots, i_{k}}(t)\, X_{i_{k}}\cdots X_{i_{1}},\quad t\in\mathbb C,\ i\in\N,
\end{equation}
which naturally appear in our Hopf algebra approach to the center problem (see Section~6) and are closely related to the classical Bell polynomials \cite{B}. We establish some recurrence relations for such polynomials (referred to as the {\em generalized displacement polynomials}) extending those of \cite[Th.6.5]{D} and describe certain important combinatorial properties of their coefficients.\smallskip

The paper is organized as follows. \smallskip

Section 2 contains the necessary background material from the theory of Hopf algebras.

Sections 3 and 4 are intended as an introduction to the areas of the shuffle and Fa\`{a} di Bruno Hopf algebras. 

Sections 5 and 6 comprise our main results and their proofs. Specifically, in Section 5.1.1 we reveal the algebraic nature of the displacement polynomials showing that they appear in the expression for the natural Hopf algebra monomorphism of the 
co-opposite of the Fa\`{a} di Bruno Hopf algebra into the shuffle Hopf algebra. As a result, we obtain some important combinatorial relations between the displacement and the Bell polynomials (see \eqref{eq5.2a}, \eqref{eq5.2b}). In Section 5.1.2, using the `generating function' for the displacement polynomials \eqref{gf}, we prove some recurrence relations for them, see Theorem \ref{recur}, partially established earlier by Devlin \cite{D} by a different method.  The first return map of equation \eqref{e1} can be factorized through the so-called Chen map \eqref{e2.16} of $\mathscr X$ into the group of characters of the shuffle Hopf algebra; this result is established in Section 5.2. In Section 5.3 we describe the group of formal centers of equation \eqref{e1} introduced earlier in \cite{Br2}. It turns out that it is a subgroup of the group of characters of the shuffle Hopf algebra isomorphic to the group of characters of the quotient of the algebra by the Hopf ideal generated by the displacement polynomials. In the same way, we describe the Lie algebra of the group of formal centers, see \eqref{eq5.7}. The Hopf algebra approach to the center problem for equations \eqref{e1} with finitely many terms is described in Section 5.4.

Finally, Section 6 is devoted to the study of the generalized displacement polynomials. We reveal their algebraic nature (showing that their values for $t\in\N$ appear as the `matrix entries' of the composition of the well-known infinite-dimensional faithful representation of  group $(G_c[[r]],\circ)$ and the first return map $\mathcal P:\mathscr X\rightarrow G_c[[r]]$, see Proposition \ref{prop6.4} and Remark \ref{rem6.6}), prove some important recurrence relations for them (Sections 6.1, 6.2.2),  establish their connection with the Bell polynomials (Section 6.2.1) and prove some combinatorial identities for their coefficients (Section 6.3).

\sect{Hopf Algebras}
In this section we collect some basic definitions and results in the area of Hopf algebras, cf., e.g., \cite{C}, \cite{CK}, \cite{GVF}, \cite{M}, \cite{S}. All objects are considered over a ground field $\mathbb K$.\medskip

{\bf (A)} An {\em associative unital algebra} is a $\mathbb K$-vector space $A$ together with a multiplication $m: A\otimes A\rightarrow A$, $m (a_1\otimes a_2)=:a_1\cdot a_2$, $a_1,a_2\in A$, such that 
\[
m\circ (m\otimes {\rm id})=m\circ ({\rm id}\otimes m)
\]
(i.e. $(a_1\cdot a_2)\cdot a_3= a_1\cdot (a_2\cdot a_3)$ for all $a_1,a_2,a_3\in A$), and a unit $\eta: \mathbb K\hookrightarrow A$ such that 
\[
\eta(k_1\cdot k_2)=\eta(k_1)\cdot \eta(k_2)\  {\rm for\ all}\  k_1,k_2\in \mathbb K\ {\rm and}\ a\cdot 1_A=a=1_A\cdot a\ {\rm for\ all}\ a\in A;
\]
here $1_A:=\eta(1_{\mathbb K})$.\medskip

{\bf (C)} A {\em coassociative counital coalgebra} is a $\mathbb K$-vector space $C$ together with a comultiplication $\Delta: C\rightarrow C\otimes C$, $\Delta(c):=\sum c_{(1)}\otimes c_{(2)}$, such that
\[
(\Delta\otimes {\rm id})\circ\Delta=({\rm id}\otimes\Delta)\circ\Delta,
\]
and a counit $\varepsilon: C\rightarrow\mathbb K$ such that
\[
(\varepsilon\otimes {\rm id}_C)\circ\Delta=({\rm id}_C\otimes\varepsilon)\circ\Delta={\rm id}_C.\smallskip
\]

{\bf (B)} A {\em bialgebra}  is a $\mathbb K$-vector space $B$ which is both an associative algebra and a coassociative coalgebra such that comultiplication $\Delta$ and counit $\varepsilon$ are algebra morphisms, i.e. for all $b_1, b_2\in B$,
\[
\Delta(b_1\cdot b_2)=\Delta(b_1)\cdot\Delta(b_2),\quad \Delta(1_H)=1_H\otimes 1_H,\quad
\varepsilon(b_1\cdot b_2)=\varepsilon(b_1)\cdot\varepsilon(b_2),\quad \varepsilon(1)=1;
\]
here the product on $B\otimes B$ is given by $(b_1\otimes b_1')\cdot (b_2\otimes b_2')=(b_1\cdot b_2)\otimes (b_1'\cdot b_2')$ for all $b_1, b_2,b_1',b_2'\in B$.\medskip

{\bf (H)} A {\em Hopf algebra}  is a bialgebra $H$ together with
a $\mathbb K$-linear map $S: H\rightarrow H$, called the {\em antipode}, such that 
\begin{equation}\label{e2.6}
m\circ (S\otimes {\rm id})\circ\Delta=\eta\circ\varepsilon=m\circ({\rm id}\otimes S)\circ \Delta
\end{equation}
and $S$ is  both an antimorphism of algebras and an antimorphism of coalgebras, i.e. for all $h_1,h_2\in H$,
\[
S(h_1\cdot h_2)=S(h_2)\cdot S(h_1),\  S(1_H)=1_H,\
\Delta(S(h_1))=(S\otimes S)(\Delta^{op}(h_1)),\ \varepsilon(S(h_1))=\varepsilon(h_1),
\]
where $\Delta^{op}=\tau\circ\Delta$, $\tau(u\otimes v)=v\otimes u$.
\begin{E}
{\rm Recall that the {\em universal enveloping algebra} $U(\mathfrak g)$ of a Lie algebra $(\mathfrak g, [\cdot,\cdot])$ is the quotient of the free associative unital algebra on the vector space $\mathfrak g$ by the two-sided ideal generated by elements of the form $x\cdot y-y\cdot x-[x,y]\cdot 1$, $x,y\in\mathfrak g$. It is a Hopf algebra with comultiplication, counit and antipode given on the generators $x\in\mathfrak g$ by
\[
\Delta(x)=x\otimes 1+1\otimes x,\qquad \varepsilon(x)=0,\qquad S(x)=-x.
\]
Note that the Hopf algebra $U(\mathfrak g)$ is {\em cocommutative} (i.e. $\Delta(u)=\Delta^{op}(u)$ for all $u\in U(\mathfrak g)$).
}
\end{E}

A bialgebra $B$ is called {\em graded} if there are $\mathbb K$-vector spaces $B_n$, $n\in\mathbb Z_+$, such that 
\begin{equation}\label{e2.7}
B=\bigoplus_{n\ge 0}B_n,\quad m(B_n\otimes B_m)\subset B_{n+m}, \quad \Delta(B_n)\subset\bigoplus_{r+s=n}B_r\otimes B_s.
\end{equation}

Elements $b\in B_n$ have degree $n$ (written, ${\rm deg}(b)=n$). 
A graded bialgebra $B$ is called {\em connected} if $B_0$ is one-dimensional, i.e. $B_0=\mathbb K\cdot 1_H$. 

Every connected graded bialgebra $B$ is a Hopf algebra with the antipode
$S$ defined by certain recursive relations and satisfying $S(B_n)\subset B_n$ for all $n\in\mathbb Z_+$.\medskip

Let $(B,m,\eta,\Delta,\varepsilon)$ be a bialgebra and $A$ be an associative unital algebra with multiplication $m_A$ and unit $\eta_A:\mathbb K\hookrightarrow A$. The vector space $L(B;A)$ of $\mathbb K$-linear maps from $B$ to $A$ inherits a canonical associative unital algebra structure with multiplication given by convolution:
\begin{equation}\label{e2.8}
\alpha *\beta:=m_A\circ (\alpha\otimes\beta)\circ\Delta,\qquad \alpha,\beta\in L(B;A),
\end{equation}
and unit $\iota:=\eta_A\circ\varepsilon$.

If $(H,m,\eta,\Delta,\varepsilon,S)$ is a commutative Hopf algebra and  $(A,m_A,\eta_A)$ is a commutative unital algebra, then the subset $G_H(A)\subset L(H,A)$ of $A$-valued characters of $H$ (i.e. unital algebra morphisms from $H$ to $A$) forms a group with respect to the convolution product $*$ with unit $\iota=\eta_A\circ\varepsilon$ and the inverse given by  $\alpha^{*-1}:=\alpha\circ S$, $\alpha\in G_H(A)$.
Thus $H$ can be regarded as the algebra of $A$-valued functions on group $G_H(A)$ equipped with pointwise multiplication, i.e. each $h\in H$ can be seen as a function on $G_H(A)$ given by $h(\alpha):=\alpha(h)$, $\alpha\in G_H(A)$. Moreover, comultiplication $\Delta$ then coincides with comultiplication on functions determined by the group law in $G_H(A)$, i.e. $\Delta(h)(\alpha,\beta)=h(\alpha *\beta)$, $h\in H$, $\alpha,\beta\in G_H(A)$.

Further, an $A$-valued infinitesimal character of $H$ is a map $\alpha\in L(H,A)$ such that
\[
\alpha(h_1\cdot h_2)=\alpha(h_1)\cdot\iota(h_2)+\iota(h_1)\cdot\alpha(h_2),\qquad h_1,h_2\in H.
\]
Since, $\iota(1_H)=1_A$, $\alpha(1_H)=0$. It follows that the set $\mathfrak g_H(A)$ of $A$-valued infinitesimal 
characters of $H$ is a Lie algebra with bracket given by
\[
[\alpha,\beta]:=\alpha*\beta-\beta*\alpha,\qquad \alpha,\beta\in \mathfrak g_H(A).\smallskip
\]

Now, assume that $H=\oplus_{n\ge 0}H_n$ is a connected graded commutative Hopf algebra and $\mathbb K$ is of characteristic zero. Then by the Milnor-Moore theorem, see, e.g., \cite[Th.\,3.8.3]{C}, $H$ is a {\em free commutative algebra} generated by homogeneous elements. Also, $G_H(A)$ is a pro-unipotent Lie group
with the Lie algebra $\mathfrak g_H(A)$ and 
the exponential map 
\[
\exp(\alpha):=\iota+\sum_{n=1}^\infty\frac{\alpha^{*n}}{n!},\qquad \alpha\in \mathfrak g_H(A),
\]
maps $\mathfrak g_H(A)$ bijectively onto $G_H(A)$.

In addition, assume that ${\rm dim}_{\mathbb K}H_n<\infty$ for all $n$.  Let $H_n^*$ be the dual space of $H_n$ over $\mathbb K$. Then $H^*=\oplus_{n\ge 0}\,H_n^*$ is the graded dual Hopf algebra with multiplication $\Delta^*$, comultiplication $m^*$, unit $\varepsilon^*$, counit $\eta^*$ and antipode $S^*$. Since $H$ is commutative, its dual $H^*$ is a cocommutative Hopf algebra. Hence, by the Milnor-Moore theorem, see, e.g., \cite[Th.\,3.8.1]{C}, $H^*$ is the universal enveloping algebra of the Lie algebra of its primitive elements
\[
{\rm Prim}\, H^*:=\{h^*\in H^* : m^*(h^*)=h^*\otimes 1_{H^*}+1_{H^*}\otimes h^*\}.
\]
Since $H^*\subset L(H;\mathbb K)$ and  $1_{H^*}=\varepsilon$,
${\rm Prim}\, H^*$ is a Lie subalgebra of the Lie algebra $\mathfrak g_H(\mathbb K)$. In fact, $\mathfrak g_H(\mathbb K)$ is the completion of ${\rm Prim}\, H^*$ in $L(H;\mathbb K)$ equipped with the adic topology induced by the grading of $H$.  

Next, the Hopf algebra operations on $H^*$ extend by continuity to similar operations on $L(H;\mathbb K)$ turning the latter into a topological Hopf algebra. (Here the extension of $\Delta^*$ coincides with the convolution product, cf. \eqref{e2.8}, and of $\varepsilon^*$ with $\iota$.) Then the group of characters $G_H({\mathbb K})\subset L(H;\mathbb K)$ coincides with the set of {\em group-like elements}, that is elements $\alpha\in L(H;\mathbb K)$ such that
\[
m^*(\alpha)=\alpha\otimes\alpha \quad {\rm and}\quad \eta^*(\alpha)=1.
\]
\sect{Shuffle Hopff Algebra} 
We refer to \cite{C}, \cite{L}, \cite{Ra}, \cite{Re}, \cite{R} and references therein for basic results in the area of shuffle Hopff algebras.
\subsection{Definition} 
Let $\mathscr A=\bigl\{\alpha_i :i\in\N\}$ be a countable alphabet. By definition, a {\em word} is an ordered sequence $\alpha_{i_1}\dots \alpha_{i_{k}}$ of (not necessarily distinct) elements from $\mathscr A$. The set of all words together with the empty word $\emptyset\, (=:1)$ is denoted by $\mathscr A^*$. Let $\mathbb K\langle \mathscr A\rangle$ be the vector space over a field $\mathbb K$ of characteristic zero freely generated by elements of $\mathscr A^*$.
To introduce multiplication on $\mathbb K\langle \mathscr A\rangle$ we invoke the following definition.\smallskip

A permutation $\sigma$ of $\{1,2,...,r + s\}$ is called a {\em shuffle} of type $(r,s)$ (denoted $\sigma\in {\rm Sh}_{r,s}$) if $\sigma^{-1}(1) < \sigma^{-1}(2) < \cdots < \sigma^{-1}(r)$ and $\sigma^{-1}(r +1)<\sigma^{-1}(r +2)<\cdots <\sigma^{-1}(r +s)$.\smallskip\footnote{The term ``shuffle'' is used because such permutations arise in riffle shuffling a deck of $r + s$ cards cut into one pile of $r$ cards and a second pile of $s$ cards.}

Now, $\mathbb K\langle \mathscr A\rangle$ is equipped with the structure of a commutative algebra by defining the {\em shuffle product} of words
\begin{equation}\label{e2.9}
\alpha_{i_1}\dots \alpha_{i_r}\shuffle \alpha_{i_{r+1}}\dots \alpha_{i_{r+s}}:=\sum_{\sigma\in {\rm Sh}_{r,s}}\alpha_{i_{\sigma(1)}}\dots \alpha_{i_{\sigma(r+s)}}.
\end{equation}
Here the empty word $1\in \mathscr A^*$ is the unit of $ \mathbb K\langle \mathscr A\rangle$, i.e., $1\shuffle w=w=w\shuffle 1$ for all $w\in \mathscr A^*$.

Next, the grading on $ \mathbb K\langle \mathscr A\rangle$ is given by ${\rm deg}(\alpha_i):=i$ ($i\in\N$), ${\rm deg}(1):=0$ so that $ \mathbb K_i\langle \mathscr A\rangle$, $i\ge 1$,
is the subspace of $ \mathbb K\langle \mathscr A\rangle$  generating by words $w=\alpha_{i_1} \dots \alpha_{i_p}\in \mathscr A^*$ such that ${\rm deg}(w):=i_1+\cdots +i_p=i$.
Then the shuffle algebra $(\mathbb K\langle \mathscr A\rangle,\shuffle)=\oplus_{i\ge 0}\,\mathbb K_i\langle \mathscr A\rangle$ is a connected graded Hopf algebra with comultiplication defined by {\em decatenation},
\begin{equation}\label{e2.10}
\Delta (\alpha_{i_1} \dots \alpha_{i_p}):=\alpha_{i_1} \dots \alpha_{i_p}\otimes 1+1\otimes \alpha_{i_1} \dots \alpha_{i_p}+\sum_{j=1}^{p-1}\alpha_{i_1} \dots \alpha_{i_j}\otimes \alpha_{i_{j+1}}\dots \alpha_{i_p},
\end{equation}
and with counit $\varepsilon : \mathbb K\langle \mathscr A\rangle\rightarrow\mathbb K$ equal to zero at each nonempty word in $\mathscr A^*$ and one at $1$.
Also, the antipode is given by the formulas
\begin{equation}\label{eq3.3}
S(\alpha_{i_1} \dots \alpha_{i_p})=(-1)^p\, \alpha_{i_p}\dots \alpha_{i_1}.
\end{equation}

Next, let us introduce the lexicographic order $\prec$ on $\mathscr A^*$ assuming initially that $\alpha_i\prec \alpha_j$ if and only if $i<j$. By cyclic permutations a word $w=\alpha_{i_1}\dots \alpha_{i_k}$ generates $k$ words $w(1),\dots ,w(k)$ with $w(1) = w$. A {\em Lyndon word} is a word $w$ such that $w(1),\dots,w(k)$ are all distinct and $w\prec w(j)$ for $j = 2,\dots ,k$. The classical result by Radford \cite{Ra} asserts that $\mathbb K\langle \mathscr A\rangle$ is a graded polynomial algebra over $\mathbb K$ in the Lyndon words as generators.
\subsection{Set of Characters}
The graded dual Hopf algebra of $\mathbb K\langle \mathscr A\rangle$ is isomorphic to the associative algebra 
$\mathbb K\langle \mathbf X\rangle$, $\mathbf X=\{X_i : i\in\N\}$,  with unit $I$ of noncommutative polynomials in $I$ and free noncommutative variables $X_i$ with coefficients in $\mathbb K$. The duality between $\mathbb K\langle \mathscr A\rangle$ and $ \mathbb K\langle \mathbf X\rangle$ is defined by putting the 
monomial basis $\mathbf X^*$ of $ \mathbb K\langle \mathbf X\rangle$ in the natural duality with words in $\mathscr A^*$.
Also, the grading on $\mathbb K\langle \mathbf X\rangle$ is given by ${\rm deg}(X_i)=i$ $(i\in\mathbb N)$, ${\rm deg}(I)=0$ (by $\mathbb K_i\langle \mathbf X\rangle\subset\mathbb K\langle \mathbf X\rangle$ we denote the subspace of elements of degree $i$), 
and comultiplication $\Delta$ is defined on the generators by 
\[
\Delta(X_i)=I\otimes X_i+X_i\otimes I,\qquad i\in\N.
\]
The set of primitive elements of $\mathbb K\langle \mathbf X\rangle$ is the free Lie algebra over $\mathbb K$ generated by $X_i$, $i\in\N$
(and $\mathbb K\langle \mathbf X\rangle$ is its universal enveloping algebra).

Further, the set $L(\mathbb K\langle \mathscr A\rangle; \mathbb K)$ of $\mathbb K$-linear functionals on $\mathbb K\langle \mathscr A\rangle$ is the completion of $\mathbb K\langle\mathbf X\rangle$ with respect to its grading.
It is naturally identified with the subalgebra of the associative algebra $\mathbb K\langle \mathbf X\rangle [[t]]$ of formal power series in $t$ with coefficients in $\mathbb K\langle \mathbf X\rangle$ consisting of series of the form
 \begin{equation}\label{e2.11}
f=f_{0}I+\sum_{i=1}^{\infty}f_i\, t^{i},\qquad f_i\in \mathbb K_i\langle \mathbf X\rangle .
\end{equation}
Thus  $L(\mathbb K\langle \mathscr A\rangle; \mathbb K)$ has the structure of a topological Hopf algebra
with comultiplication extending by linearity and continuity  comultiplication $\Delta$ on $\mathbb K\langle \mathbf X\rangle$. In turn, the group
$G_{\mathbb K\langle \mathscr A\rangle}(\mathbb K)$ of $\mathbb K$-valued characters of $\mathbb K\langle \mathscr A\rangle$ is the set of group-like elements of $L(\mathbb K\langle \mathscr A\rangle; \mathbb K)$, and the Lie algebra $\mathfrak g_{\mathbb K\langle \mathscr A\rangle}(\mathbb K)$ of infinitesimal characters of $\mathbb K\langle \mathscr A\rangle$ is the set of primitive elements of $L(\mathbb K\langle \mathscr A\rangle; \mathbb K)$. It
consists of series $g=\sum_{i=1}^\infty g_i\, t^i$ with $g_i\in \mathbb K_i\langle \mathbf X\rangle\cap \mathfrak g_{\mathbb K\langle \mathscr A\rangle}(\mathbb K)$  {\em Lie elements} of degree $i$ (i.e. having the  
form
\begin{equation}\label{e2.14}
g_i=\sum_{i_{1}+\dots +i_{k}=i}d_{i_{1},\dots, i_{k}}[X_{i_{1}},[X_{i_{2}},[\ \cdots , [X_{i_{k-1}},X_{i_{k}}]\cdots \ ]]] ;
\end{equation}
here $[X,Y]:=XY-YX$ and the term with $i_{k}=i$ is $d_{i}X_{i}$).

Hence, the exponential map $\exp(g):=I+\sum_{n=1}^\infty\frac{g^n}{n!}$, $g\in L(\mathbb K\langle \mathscr A\rangle; \mathbb K)$, maps $\mathfrak g_{\mathbb K\langle \mathscr A\rangle}(\mathbb K)$ bijectively onto $G_{\mathbb K\langle \mathscr A\rangle}(\mathbb K)$.
\begin{R}
{\rm From \cite[Th.\,3.2]{M-KO} one obtains that
\begin{equation}\label{e3.5}
{\rm dim}\bigl(\mathbb K_i\langle \mathbf X\rangle\cap \mathfrak g_{\mathbb K\langle \mathscr A\rangle}(\mathbb K)\bigr)=\frac{1}{i}\sum_{d|i}(2^{\ \!i/d}-1)\cdot\mu(d),
\end{equation}
where the sum is taken over all numbers $d\in\N$ that divide $i$, and $\mu:\N\rightarrow\{-1,0,1\}$ is the M\"{o}bius function defined as follows.
If $d$ has a prime factorization
$$
d=p_{1}^{n_{1}}p_{2}^{n_{2}}\cdots p_{q}^{n_{q}},\ \ \ n_{i}>0,
$$
then
$$
\mu(d)=
\left\{
\begin{array}{cc}
1&{\rm for}\ d=1\\
(-1)^{q}&\quad {\rm if\ all}\ n_{i}=1\\
0&\ {\rm otherwise}.
\end{array}
\right.
$$

The basis of $\mathfrak g_{\mathbb K\langle \mathscr A\rangle}(\mathbb K)$ can be constructed by means of Lyndon words of $\mathbf X^*$, see \cite{Ly}.
}
\end{R}
\subsection{Semigroup of Paths}
Let us consider the set $\mathscr X$ of coefficients of equation \eqref{e1} as a nonassociative semigroup with the operations given for $a=(a_{1},a_{2},\dots)$ and $b=(b_{1},b_{2},\dots)$ in $\mathscr X$ by
$$
a*b=(a_{1}*b_{1},a_{2}*b_{2},\dots)\in \mathscr X\ \ \ {\rm and}\ \ \
a^{-1}=(a_{1}^{-1},a_{2}^{-1},\dots)\in \mathscr X,
$$
where for $i\in\N$
$$
(a_{i}*b_{i})(x)=\left\{
\begin{array}{ccc}
2b_{i}(2x)&{\rm if}&0\leq x\leq T/2,\\
2a_{i}(2x-T)&{\rm if}&T/2<x\leq T
\end{array}
\right.
$$
and
$$
a_{i}^{-1}(x)=-a_{i}(T-x),\ \ \ 0\leq x\leq T.
$$

Let $\mathbb C^{\infty}$ be the vector space of sequences of complex numbers $(c_{1},c_{2},\dots)$ equipped with the product topology.
For $a=(a_{1},a_{2},\dots)\in \mathscr X$ by $\widetilde a= (\widetilde a_{1},
\widetilde a_{2},\dots): I_{T}\rightarrow\mathbb C^{\infty}$, 
$\widetilde a_{k}(x):=\int_{0}^{x}a_{k}(t)\ \!dt$ for all $k\in\N$,
we denote a path in $\mathbb C^{\infty}$ starting at 0.
The one-to-one map $a\mapsto\widetilde a$ sends the product $a*b$ to the product of paths $\widetilde a\circ\widetilde b$, that is the path obtained by translating $\widetilde a$ so that its beginning meets the end of $\widetilde b$ and then forming the composite path. Similarly, $\widetilde {a^{-1}}$ is the path obtained by translating $\widetilde a$ so that its end meets 0 and then taking it with the opposite orientation. \smallskip

For $a=(a_1,a_2,\dots)\in \mathscr X$ let us consider the {\em basic iterated integrals}
\[
I_{i_{1},\dots, i_{k}}(a):=\int\cdots\int_{0\leq s_{1}\leq\cdots\leq s_{k}\leq T}a_{i_{k}}(s_{k})\cdots a_{i_{1}}(s_{1})\ \!ds_{k}\cdots ds_{1}
\]
(for $k=0$ we assume that this equals $1$). By the Ree shuffle formula \cite{R} the linear space over $\mathbb C$ generated by all such functions on $\mathscr X$ is an algebra (i.e. it is closed under the pointwise multiplication of functions on $\mathscr X$) denoted by $\mathcal I_{\mathbb C}(\mathscr X)$. Also, the correspondence $\mathscr A^*\ni \alpha_{i_k}\dots \alpha_{i_1}\mapsto I_{i_{1},\dots, i_{k}}\in \mathcal I_{\mathbb C}(\mathscr X)$ for all possible indices $i_1,\dots, i_k$ extends (by linearity) to an isomorphism of commutative algebras $\mathbb C\langle\mathscr A\rangle\rightarrow \mathcal I_{\mathbb C}(\mathscr X)$ whose transpose determines  the Chen map $E:\mathscr X\rightarrow L(\mathbb C\langle \mathscr A\rangle; \mathbb C)$ (see \eqref{e2.11}),
\begin{equation}\label{e2.16}
E(a):=E(a)=I+\sum_{i=1}^{\infty}\left(\ \!\sum_{i_{1}+\cdots +i_{k}=i}I_{i_{1},\dots, i_{k}}(a)X_{i_{k}}\cdots X_{i_{1}}\right)t^{i},
\end{equation}
that sends $\mathscr X$ into $G_{\mathbb C\langle\mathscr A\rangle}(\mathbb C)$ and satisfies (cf. \cite[Th.~6.1]{Ch})
\begin{equation}\label{e2.17}
E(a*b)=E(a)\cdot E(b),\qquad E(a^{-1})=E(a)^{-1}\quad {\rm for\ all}\quad a,b\in \mathscr X.
\end{equation}
Group $G_{\mathbb C\langle\mathscr A\rangle}(\mathbb C)$ has the natural structure of a complete metrizable separable topological group and $E(\mathscr X)$ is its dense subgroup, cf. \cite[Sect.\,2.3.2]{Br2}. 

Next, $\mathscr U:=E^{-1}(I)\subset\mathscr C$ and is called the {\em set of universal centers} of equation \eqref{e1}. Its basic properties are described in \cite[Sect.\,2.2]{Br2}, \cite{BY}, \cite{Br4} (see also references therein). 
For instance, if all coordinates of $a$ are real-valued $C^1$ functions on $I_T$ and the image of the path $\widetilde a: I_T\rightarrow\mathbb C^\infty$ is a piecewise smooth curve, then $a\in\mathscr U$ if and only if $\widetilde a=\widetilde a_1\circ\widetilde a_2$, where $\widetilde a_2: I_T\rightarrow\mathcal T$ is a closed path into a compact metric tree $\mathcal T$ and $\widetilde a_1\in C(\mathcal T,\mathbb C^\infty)$.

The set of centers $\mathscr C$ is closed under the semigroup operations on $\mathscr X$. In particular, $\widehat{\mathscr C}:=E(\mathscr C)$  and its closure in $G_{\mathbb C\langle\mathscr A\rangle}(\mathbb C)$ denoted by $\widehat{\mathscr C}_f$ are subgroups of $G_{\mathbb C\langle\mathscr A\rangle}(\mathbb C)$ called the groups of centers and of formal centers of equation \eqref{e1}. 

\sect{Fa\`{a} di Bruno Hopf Algebra}
For the main results and references on the Fa\`{a} di Bruno Hopf algebra see, e.g.,  \cite{FM}, \cite{FG}, \cite{GVF}, \cite{JR}.

The algebra, denoted by $\mathcal H_{{\rm FdB}}(\mathbb K)$, where $\mathbb K$ is a field of characteristic zero, is named after the classical formula computing the $n$th derivative of the composition of two functions of one variable attributed to Fa\`{a} di Bruno \cite{FdB}.  In our approach to the center problem we work with the co-opposite Hopf algebra $\mathcal H_{{\rm FdB}}^{cop}(\mathbb K)$ with comultiplication differed from that of  $\mathcal H_{{\rm FdB}}(\mathbb K)$ by the flip, cf. Section 2\,(H).
By definition,
$\mathcal H_{{\rm FdB}}^{cop}(\mathbb K):=\mathbb K[t_1,t_2,\dots]$ is the algebra of polynomials with coefficients in $\mathbb K$ in commutative variables $t_i$, $i\in\N$, with coproduct and counit defined on the generators by the formulas
\begin{equation}\label{e2.18}
\begin{array}{l}
\displaystyle
\Delta_{{\rm FdB}}^{op}(t_i):=\sum_{j=0}^i t_j\otimes\left( \sum_{\stackrel{l_0+l_1+\cdots +l_i=j+1}{_{l_1+2l_2+\cdots +il_i=i-j}}}\frac{(j+1)!}{l_0!\, l_1!\cdots l_i!}t_1^{l_1}\cdots t_i^{l_i}\right),\\
\\
\displaystyle \varepsilon(t_i)=\delta_{i0} \quad i\in\mathbb Z_+;\quad {\rm here}\quad t_0:=1.
\end{array}
\end{equation}
The expression in brackets can be written as 
\[
\frac{(j+1)!}{(i+1)!}B_{i+1,j+1}\bigl(1,2!\,t_1, 3!\,t_2,\dots, (i-j+1)!\,t_{i-j}\bigr),
\] 
where
\begin{equation}\label{e2.19}
B_{r,s}(t_1,\dots, t_l):=\sum_{\stackrel{k_1+k_2+\cdots +k_l=s}{_{k_1+2k_2+\cdots +lk_l=r}}}\frac{r!}{k_1!\cdots k_l!}\left(\frac{t_1}{1!}\right)^{k_1}\cdots \left(\frac{t_l}{l!}\right)^{k_l},\quad l=r-s+1,
\end{equation}
are the {\em Bell polynomials} \cite{B}.

The antipode $S_{\rm FdB}$ of $\mathcal H_{{\rm FdB}}^{cop}(\mathbb K)$ (and of $\mathcal H_{{\rm FdB}}(\mathbb K)$ as well) is given by the formulas 
\begin{equation}\label{e2.20}
S_{\rm FdB}(t_i)=\frac{1}{(i+1)!}\sum_{j=1}^{i}(-1)^j B_{i+j,j}\bigl(0,2!\,t_1,3!\,t_2,\dots, (i+1)!\, t_{i}\bigr), \quad i\in\N.
\end{equation}
Hopf algebra $\mathcal H_{{\rm FdB}}^{cop}(\mathbb K)$ is connected and graded with generators $t_i$ having degree $i$, $i\in\N$.  Let $\mathcal H_{{\rm FdB}}^{i}(\mathbb K)\subset \mathcal H_{{\rm FdB}}^{cop}(\mathbb K)$ be the subspace of elements of degree $i$. By the Milnor-Moore theorem the dual Hopf algebra $\mathcal H_{{\rm FdB}}^{cop}(\mathbb K)^*=\oplus_{i\ge 0}\,\mathcal H_{{\rm FdB}}^{i}(\mathbb K)^*$ is the universal enveloping algebra (equipped with the convolution product) of the Lie algebra ${\rm Prim}\, \mathcal H_{{\rm FdB}}^{cop}(\mathbb K)^*$ of its primitive elements. The latter has a basis $t_i'$, $i\in\N$, with ${\rm deg}(t_i')=i$ such that
\begin{equation}\label{eq4.4}
t_i'(t_j)=\delta_{ij}\quad {\rm and}\quad t_i'(t_{i_1}\cdots t_{i_{k}})=0\quad {\rm for}\quad k\ge 2
\end{equation}
which satisfies the following relations
\begin{equation}\label{eq4.5}
[t_i',t_j']:=t_i'*t_j'-t_j'*t_i'=(i-j)\,t_{i+j}'\quad {\rm for\ all}\quad i,j\in\N.
\end{equation}
\begin{R}
{\rm
It is worth noting that ${\rm Prim}\, \mathcal H_{{\rm FdB}}^{cop}(\mathbb C)^*$ is isomorphic to the Lie algebra $W_1(1)={\rm span}_{\mathbb C}\{e_n:=t^{n+1}\frac{d}{dt},\, n\in\N \}$, the nilpotent part of the {\em Witt algebra} of complex formal vector fields on $\mathbb R$.}
\end{R}
Next, the Lie algebra $\mathfrak g_{\mathcal H_{{\rm FdB}}^{cop}(\mathbb K}(\mathbb K)$ of $\mathbb K$-valued infinitesimal characters of $\mathcal H_{{\rm FdB}}^{cop}(\mathbb K)$ is naturally identified with the set of formal power series of the form $g(t):=\sum_{i=1}^\infty (d_i\, t_i') t^i$, $d_i\in\mathbb K$,
with the bracket extending by linearity and the adic continuity the bracket on ${\rm Prim}\,\mathcal H_{{\rm FdB}}^{cop}(\mathbb K)^*$. In turn, the set
$L(\mathcal H_{{\rm FdB}}^{cop}(\mathbb K);\mathbb K)$ of $\mathbb K$-linear functionals on $\mathcal H_{{\rm FdB}}^{cop}(\mathbb K)$ is the completion of $\mathcal H_{{\rm FdB}}^{cop}(\mathbb K)^*$ with respect to its grading. It is identified with the set of formal power series in $t$ with coefficients in $\mathcal H_{{\rm FdB}}(\mathbb K)^*$ of the form
\begin{equation}\label{eq4.6}
f(t)=\sum_{i=0}^\infty f_i\, t^{i},\qquad f_i\in \mathcal H_{{\rm FdB}}^i(\mathbb K)^*,\quad i\in \mathbb Z_+,
\end{equation}
with multiplication $*$ extending by linearity and continuity the convolution product on $\mathcal H_{{\rm FdB}}^{cop}(\mathbb K)^*$.  By definition, $f(t)(t_i)=f_i(t_i)\,t^i$ for all $i\in\mathbb Z_+$ (here $t_0:=1$).

Further, the set of characters $G_{\mathcal H_{{\rm FdB}}^{cop}(\mathbb K)}(\mathbb K)\subset L(\mathcal H_{{\rm FdB}}^{cop}(\mathbb K);\mathbb K)$ of $\mathcal H_{{\rm FdB}}^{cop}(\mathbb K)$ is the image of ${\rm Prim}\, \mathcal H_{{\rm FdB}}^{cop}(\mathbb K)^*$ under the exponential map.
Each $f\in G_{\mathcal H_{{\rm FdB}}^{cop}(\mathbb K)}(\mathbb K)$ is uniquely determined by its values on generators. Moreover, map $\Theta: \bigl(G_{\mathcal H_{{\rm FdB}}^{cop}(\mathbb K)},*\bigr)\rightarrow \bigl(G_{\mathbb K}[[r]],\circ\bigr)$ (- the group with respect to the composition of series of formal power series of the form $r+\sum_{i=1}^\infty c_i\, r^{i+1}$ with all $c_i\in\mathbb K$) given by 
\begin{equation}\label{eq4.7}
\Theta(f)(r):=r+\sum_{i=1}^\infty f_i(t_i)\, r^{i+1},\qquad f\in G_{\mathcal H_{{\rm FdB}}^{cop}(\mathbb K)},
\end{equation}
is an isomorphism of topological groups.
 \sect{Center Problem}
 In this section we describe the main ingredients of the Hopf algebra approach to the center problem for equation \eqref{e1}. Some of our results are already proved in \cite{Br2}, so we just adapt them to the new setting.
 \subsection{Displacement Polynomials} 
 \subsubsection{{\bf Embedding of $\mathcal H_{{\rm FdB}}^{cop}(\mathbb K)$ into $\mathbb K\langle \mathscr A\rangle$}}
 Consider the canonical surjective morphism of Lie algebras $\rho_\mathbb K: {\rm Prim}\, \mathbb K\langle\mathbf X\rangle\rightarrow {\rm Prim}\, \mathcal H_{{\rm FdB}}^{cop}(\mathbb K)^*$
 defined on the generators by  
 \[
 \rho_\mathbb K(X_i)=t_i',\quad i\in\N.
 \]
 By the universal property of enveloping algebras it extends to a surjective morphism of graded Hopf algebras
 $\rho_\mathbb K : \mathbb K\langle\mathbf X\rangle\rightarrow \mathcal H_{{\rm FdB}}^{cop}(\mathbb K)^*$ whose transpose is a monomorphism of graded Hopf algebras $\rho_\mathbb K^*: \mathcal H_{{\rm FdB}}^{cop}(\mathbb K)\rightarrow \mathbb K\langle \mathscr A\rangle$.
 \begin{Th}\label{te5.1}
The values of $\rho_\mathbb K^*$ on the generators are the displacement polynomials, i.e.
 \begin{equation}\label{eq5.1}
\rho_\mathbb K^*(t_i)=\mathcal P_i(\alpha_1,\dots, \alpha_i),\quad i\in\N;
\end{equation}
here the product of words in $\mathscr A^*$ is defined by concatenation.
\end{Th}
\begin{proof}
For $k\ge 2$ and $\bar X:=X_{i_k}\cdots  X_{i_1}$, ${\rm deg}\,(\bar X):=i_1+\cdots + i_k$, we have
\[
\begin{array}{lr}
\displaystyle
\bar X(\rho_\mathbb K^*(t_i))=\rho_\mathbb K(\bar X)(t_i)=(t'_{i_k}*\cdots  *t'_{i_1})(t_i)
 = m_{\mathbb K}\circ\left(\bigl(t'_{i_k}*\cdots *t'_{i_{2}}\bigr)\otimes t_{i_1}'\right)\circ\Delta_{{\rm FdB}}^{op}(t_i)\medskip\\
\displaystyle =\sum_{j=0}^i \bigl(t'_{i_k}*\cdots *t'_{i_{2}}\bigr)(t_j)\cdot \frac{(j+1)!}{(i+1)!}\cdot t_{i_1}'\!\!\left(B_{i+1,j+1}\bigl(1,2!\,t_1, 3!\,t_2,\dots, (i-j+1)!\,t_{i-j}\bigr)\right).
\end{array}
\]
The last factor is nonzero if $i_1\le i$ and $j=i-i_1$, see \eqref{e2.18}, \eqref{e2.19}. In this case the previous expression implies that
\[
(X_{i_k}\cdots  X_{i_1})(\rho_\mathbb K^*(t_i))=(i-i_1+1)\cdot (X_{i_k}\cdots  X_{i_{2}})(\rho_\mathbb K^*(t_{i-i_1})).
\]
From the above recursive relations we obtain
\[
\bar X(\rho_\mathbb K^*(t_i))=(i-i_1+1)(i-i_1-i_2+1)\cdots 1\cdot \delta_{i\, {\rm deg}(\bar X)}.
\]
Also, $X_i(\rho_\mathbb K^*(t_i))=t_i'(t_i)=1$.
From here passing to the dual basis $\mathscr A^*\subset  \mathbb K\langle \mathscr A\rangle$ we get 
\[
\rho_\mathbb K^*(t_i)=\sum_{i_1+\dots +i_k=i} p_{i_1,\dots, i_k}(i)\,\alpha_{i_k}\dots\alpha_{i_1}=:\mathcal P_i(\alpha_1,\dots,\alpha_i).
\]
\end{proof}
\begin{R}
{\rm (1) Since $\rho_\mathbb K^*: \mathcal H_{{\rm FdB}}^{cop}(\mathbb K)\rightarrow \mathbb K\langle \mathscr A\rangle$ is a morphism of Hopf algebras,
\begin{equation}\label{e2.25}
\rho_\mathbb K^*\circ S_{\rm FdB}=S_{\mathbb K\langle\mathscr A\rangle}\circ\rho_\mathbb K^*\quad  {\rm and}\quad  (\rho_\mathbb K^*\otimes\rho_\mathbb K^*)\circ\Delta_{\rm FdB}^{op}=\Delta_{\mathbb K\langle \mathscr A\rangle}\circ \rho_\mathbb K^*.
\end{equation}
These and Theorem \ref{te5.1} lead to the following combinatorial relations in the shuffle Hopf algebra $\mathbb K\langle\mathscr A\rangle$ between the displacement and the Bell polynomials:
\begin{equation}\label{eq5.2a}
\bar{\mathcal P_i}:=\frac{1}{(i+1)!}\sum_{j=1}^{i}(-1)^j B_{i+j,j}\bigl(0,2!\,\mathcal P_1,3!\,\mathcal P_2,\dots, (i+1)!\, \mathcal P_{i}\bigr),
\end{equation}
where
\[
\bar{\mathcal P_i}(\alpha_1,\dots, \alpha_i):=(S_{\mathbb K\langle\mathscr A\rangle}\circ\rho_\mathbb K^*)(t_i)=\sum_{i_1+\dots +i_k=i} (-1)^k\cdot p_{i_k,\dots, i_1}(i)\,\alpha_{i_k}\dots\alpha_{i_1},\qquad i\in\N;
\]
\begin{equation}\label{eq5.2b}
\Delta_{\mathbb K\langle\mathscr A\rangle}(\mathcal P_i)=\sum_{j=0}^i \mathcal P_j\otimes \frac{(j+1)!}{(i+1)!}B_{i+1,j+1}\bigl(1,2!\,\mathcal P_1, 3!\,\mathcal P_2,\dots, (i-j+1)!\,\mathcal P_{i-j}\bigr),
\end{equation}
where $\mathcal P_0:=1$ and for $i\in\N$,
\[
\begin{array}{l}
\displaystyle
\Delta_{\mathbb K\langle \mathscr A\rangle}(\mathcal P_i)(\alpha_1,\dots,\alpha_i)\medskip \\
\displaystyle :=\sum_{i_1+\dots +i_k=i} p_{i_1,\dots, i_k}(i)\cdot\left(1\otimes \alpha_{i_k}\dots\alpha_{i_1}+\alpha_{i_k}\dots\alpha_{i_1}\otimes 1+\sum_{j=1}^{k}\alpha_{i_k} \dots \alpha_{i_{j+1}}\otimes \alpha_{i_{j}}\dots \alpha_{i_1}\right).
\end{array}
\]
(2) It is worth noting that the displacement polynomials $\mathcal P_i\in\mathbb K\langle\mathscr A\rangle $, $i\in\mathbb Z_+$, are algebraically independent over $\mathbb K$.
}
\end{R}
\subsubsection{{\bf Recurrence Relations for the Displacement Polynomials}} We set $\mathcal P_0:=I\in \mathbb K\langle\mathbf X\rangle$.
\begin{Th}\label{recur}
The following relations hold for all $n\in\N$:
\begin{equation}\label{devlin}
\mathcal P_n(X_1,\dots, X_n)=\sum_{i=1}^n (n-i+1)\cdot \mathcal P_{n-i}(X_1,\dots, X_{n-i})\cdot X_{i}.
\end{equation}
\end{Th}
\begin{proof}
First, we show that 
\begin{equation}\label{gf}
\mathscr P_\mathbb K=\bigl(1\otimes I-\sum_{j=1}^\infty t_j'\otimes X_j\, t^j\bigr)^{-1}\in \mathcal H_{{\rm FdB}}^{cop}(\mathbb K)^*\otimes L(\mathbb K\langle\mathscr A\rangle;\mathbb K),
\end{equation}
is the `generating function' for the displacement polynomials.

Indeed, by the definition,
$$
\mathscr P_\mathbb K=1\otimes I+\sum_{i=1}^\infty\left(\sum_{j=1}^\infty t_j'\otimes X_j\, t^j\right)^i=1\otimes I+\sum_{i=1}^\infty\left(\sum_{i_1+\cdots+i_k=i}t'_{i_k}*\cdots  *t'_{i_1}\otimes X_{i_k}\cdots  X_{i_1}\right)\, t^i.
$$
In particular, due to the computation of Theorem \ref{te5.1},
\[
\mathscr P_\mathbb K(t_n):=1(t_n)\cdot I+\sum_{i=1}^\infty\left(\sum_{i_1+\cdots+i_k=i}(t'_{i_k}*\cdots  *t'_{i_1})(t_n)\otimes X_{i_k}\cdots  X_{i_1}\right)\, t^i=\mathcal P_n(X_1,\dots, X_n)\, t^n.
\]

Next, we rewrite the expression for $\mathscr P_\mathbb K$ as follows
\[
\mathscr P_\mathbb K=1\otimes I+ \mathscr P_\mathbb K\times \sum_{j=1}^\infty t_j'\otimes X_j\, t^j;
\]
here $\times $ denotes the product in the algebra $\mathcal H_{{\rm FdB}}^{cop}(\mathbb K)^*\otimes L(\mathbb K\langle\mathscr A\rangle;\mathbb K)$. 

From this presentation we obtain
\[
\begin{array}{l}
\displaystyle
\mathscr P_\mathbb K(t_n)=\left(\,\sum_{i=1}^n \left(\sum_{i_1+\cdots+i_k=n-i}t'_{i_k}*\cdots  *t'_{i_1}\otimes X_{i_k}\cdots  X_{i_1}\right)\, t^{n-i}\times \bigl(t_{i}'\otimes X_{i}\, t^{i}\bigr)\right)(t_n)\medskip\\
\displaystyle \qquad \ \ \ = \sum_{j=1}^n \left(\sum_{i_1+\cdots+i_k=n-i}(t'_{i_k}*\cdots  *t'_{i_1} *t_{i}')(t_n)\otimes X_{i_k}\cdots  X_{i_1}\cdot X_{i}\right)\, t^n\medskip\\
\displaystyle \qquad \ \ \ =\sum_{i=1}^n \left(\sum_{i_1+\cdots+i_k=n-i}(t'_{i_k}*\cdots  *t'_{i_1})(t_{n-i})\cdot (n-i+1)\otimes X_{i_k}\cdots  X_{i_1}\cdot X_{i}\right)\, t^n\medskip\\
\displaystyle \qquad \ \ \ =\sum_{i=1}^n (n-i+1)\cdot \mathcal P_{n-i}(X_1,\dots, X_{n-i})\cdot X_{i}\cdot t^n.
\end{array}
\]
Comparing the latter expression for $\mathscr P_\mathbb K(t_n)$ with the former one we attain the required recurrence relations.
\end{proof}
\subsection{First Return Map}
 Consider equation (\ref{e1}) corresponding to
$a=(a_{1}, a_{2},\dots)\in \mathscr X$:
\begin{equation}\label{eq5.2}
\frac{dv}{dx}=\sum_{i=1}^{\infty}a_{i}(x) v^{i+1},\ \ \ x\in I_{T}.
\end{equation}
We relate to this equation the following one:
\begin{equation}\label{e3.3}
\frac{dH}{dx}=\left(\,\sum_{i=1}^{\infty} a_{i}(x)\, t_i'\, t^{i}\right) H,\ \ \ x\in I_{T}.
\end{equation}
Solving (\ref{e3.3}) by Picard iteration we get a solution $H_{a}: I_{T}\rightarrow L\bigl(\mathcal H_{{\rm FdB}}^{cop}(\mathbb C);\mathbb C\bigr)$, $H_{a}(0)=1$,  whose coefficients in the series expansion are Lipschitz functions on $I_{T}$. One easily checks that (cf. \eqref{e2.16})
\begin{equation}\label{eq5.4}
H_a(T)=1+\sum_{i=1}^{\infty}\left(\ \!\sum_{i_{1}+\cdots +i_{k}=i}I_{i_{1},\dots, i_{k}}(a)\,t'_{i_{k}}*\cdots *t'_{i_{1}}\right)t^{i}=(\rho_\mathbb C\circ E)(a).
\end{equation}
 In particular, $H_a(T)\in G_{\mathcal H_{{\rm FdB}}^{cop}(\mathbb C)}(\mathbb C)$. Thus, the following result for the first return map $\mathcal P(a)$ of \eqref{eq5.2} holds:
 \begin{Th}\label{te5.2}
 \begin{equation}\label{eqte5.2}
\mathcal P(a)=\Theta\bigl(H_a(T)\bigr)=(\Theta\circ \rho_\mathbb C\circ E)(a).
 \end{equation}
 \end{Th}
 \begin{proof}
 Using the computation of the proof of Theorem \ref{te5.1} we obtain, see \eqref{e4},
 \[
 \begin{array}{l}
 \displaystyle
 \Theta\bigl(H_a(T)\bigr)(r)=r+\sum_{i=1}^{\infty}\left(\ \!\sum_{i_{1}+\cdots +i_{k}=i}I_{i_{1},\dots, i_{k}}(a)\,(t'_{i_{k}}*\cdots *t'_{i_{1}})(t_i)\right)r^{i+1}\medskip\\
 \displaystyle
 \qquad \qquad \qquad  =r+\sum_{i=1}^{\infty}\left(\ \!\sum_{i_{1}+\cdots +i_{k}=i}p_{i_1,\dots, i_k}(i)\cdot I_{i_{1},\dots, i_{k}}(a)\right)r^{i+1}=:\mathcal P(a)(r).
 \end{array}
 \]
 \end{proof}
 \subsection{Group of Formal Centers}
 Let $\mathcal J_\mathbb K\subset \mathbb K\langle\mathscr A\rangle$ be the ideal generated by the displacement polynomials $\mathcal P_i$, $i\in\N$ (cf. \eqref{eq5.1}).  Due to \eqref{e2.25}, $\mathcal J_\mathbb K$ is a {\em graded Hopf ideal} of $\mathbb K\langle\mathscr A\rangle$, i.e. $\mathcal J_\mathbb K=\oplus_{i\ge 1}\,\mathbb K_i\langle\mathscr A\rangle\cap \mathcal J_\mathbb K$. Thus
 the quotient $\mathbb K\langle\mathscr A\rangle/\mathcal J_\mathbb K=\oplus_{i\ge 0}\, \mathbb K_i\langle\mathscr A\rangle/\bigl(\mathbb K_i\langle\mathscr A\rangle\cap \mathcal J_\mathbb K\bigr)$ is a connected graded Hopf algebra. The dual Hopf algebra $\bigl(\mathbb K\langle\mathscr A\rangle/\mathcal J_\mathbb K\bigr)^*$ is the connected graded Hopf subalgebra of $\mathbb K\langle\mathbf X\rangle$ consisting of elements vanishing on $\mathcal J_\mathbb K$. By the definition of $\mathcal J_\mathbb K$, 
 \[
 \bigl(\mathbb K\langle\mathscr A\rangle/\mathcal J_\mathbb K\bigr)^*=\mathbb K_0\langle\mathbf X\rangle\oplus {\rm Ker}\,\rho_{\mathbb K}.
 \]
 Hence, $\bigl(\mathbb K\langle\mathscr A\rangle/\mathcal J_\mathbb K\bigr)^*\cap \mathbb K_i\langle\mathbf X\rangle$, $i\ge 1$, consists of elements
 \begin{equation}\label{eq5.6}
f_i=\sum_{i_{1}+\cdots +i_{k}=i}c_{i_{1},\dots, i_{k}}X_{i_{k}}\cdots X_{i_{1}}\quad {\rm such\ that}\quad \sum_{i_{1}+\cdots +i_{k}=i}p_{i_1,\dots, i_k}(i)\cdot c_{i_{1},\dots, i_{k}}=0.
 \end{equation}
 
 Next, by the Milnor-Moore theorem $\bigl(\mathbb K\langle\mathscr A\rangle/\mathcal J_\mathbb K\bigr)^*$ is the universal enveloping algebra of the Lie algebra ${\rm Prim}\, \bigl(\mathbb K\langle\mathscr A\rangle/\mathcal J_\mathbb K\bigr)^*$ of its primitive elements. By definition, 
 \[
 {\rm Prim} \bigl(\mathbb K\langle\mathscr A\rangle/\mathcal J_\mathbb K\bigr)^*={\rm Prim}\, \mathbb K\langle\mathbf X\rangle\cap {\rm Ker}\,\rho_{\mathbb K}.
 \]  
 An explicit computation using that $[t_i',t_j']=(i-j)\,t_{i+j}'$ for all $i,j\in\N$
 (cf. \cite[Sect.\,3.2.2]{Br2} for similar arguments) shows that 
${\rm Prim} \bigl(\mathbb K\langle\mathscr A\rangle/\mathcal J_\mathbb K\bigr)^*\cap \mathbb K_i\langle\mathbf X\rangle$
 consists of elements  
 $$
g_i=\sum_{i_{1}+\dots +i_{k}=i}c_{i_{1},\dots, i_{k}}[X_{i_{k}},[X_{i_{k-1}},[\ \cdots , [X_{i_{2}},X_{i_{1}}]\cdots \ ]]]
$$
such that
\begin{equation}\label{eq5.7}
\begin{array}{c}
\displaystyle 
\rho_\mathbb K(g_i)=\sum_{i_{1}+\cdots + i_{k}=i}c_{i_{1},\dots, i_{k}}\cdot\gamma_{i_{1},\dots, i_{k}}=0,\quad {\rm where}\quad  \gamma_{i}=1\ \ \ {\rm and}\medskip\\
\gamma_{i_{1},\dots, i_{k}}=
(i_{2}-i_{1})(i_3- i_{2}-i_{1})\cdots
(i_k-i_{k-1}-\cdots -i_{1})\ \ \ {\rm for}\ \ \ k\geq 2.
\end{array}
\end{equation}

In turn, the Lie algebra $\mathfrak g_{\mathbb K\langle\mathscr A\rangle/\mathcal J_\mathbb K}(\mathbb K)$ of infinitesimal characters of $\mathbb K\langle\mathscr A\rangle/\mathcal J_\mathbb K$ is the closure of ${\rm Prim}\, \bigl(\mathbb K\langle\mathscr A\rangle/\mathcal J_\mathbb K\bigr)^*$ in ${\rm Prim}\, \mathbb K\langle\mathbf X\rangle$, i.e. it consists of series $\sum_{i=1}^\infty g_i\, t^i$ with $g_i\in {\rm Prim} \bigl(\mathbb K\langle\mathscr A\rangle/\mathcal J_\mathbb K\bigr)^*\cap \mathbb K_i\langle\mathbf X\rangle$, $i\in\N$. Also,
the group of characters of $\mathbb K\langle\mathscr A\rangle/\mathcal J_\mathbb K$ is a subgroup of $G_{\mathbb K\langle\mathscr A\rangle}(\mathbb K)$ consisting of characters vanishing on $\mathcal J_\mathbb K$, i.e. 
\[
G_{\mathbb K\langle\mathscr A\rangle/\mathcal J_\mathbb K}(\mathbb K)=G_{\mathbb K\langle\mathscr A\rangle}(\mathbb K)\cap {\rm Ker}\,\rho_\mathbb K.
\]
In particular, if $\mathbb K=\mathbb C$, then, due to Theorem \ref{te5.2}, $G_{\mathbb C\langle\mathscr A\rangle/\mathcal J_\mathbb C}(\mathbb C)=\widehat{\mathscr C}_f$, the group of formal centers of equation \eqref{e1}.

By definition, $\mathfrak g_{\mathbb K\langle\mathscr A\rangle/\mathcal J_\mathbb K}(\mathbb K)$ is the Lie algebra of the Lie group $G_{\mathbb K\langle\mathscr A\rangle/\mathcal J_\mathbb K}(\mathbb K)$ and the exponential map  maps $\mathfrak g_{\mathbb K\langle\mathscr A\rangle/\mathcal J_\mathbb K}(\mathbb K)$ homeomorphically onto $\mathfrak g_{\mathbb K\langle\mathscr A\rangle/\mathcal J_\mathbb K}(\mathbb K)$.

Note that the linear continuous map $\pi: \mathfrak g_{\mathbb K\langle\mathscr A\rangle}(\mathbb K)\rightarrow  \mathfrak g_{\mathbb K\langle\mathscr A\rangle/\mathcal J_\mathbb K}(\mathbb K)$,
\begin{equation}\label{eq5.8}
\pi\left(\sum_{i=1}^\infty g_i\, t^i\right):=\sum_{i=1}^\infty \bigl(g_i-\rho_\mathbb K(g_i)\cdot X_i\bigr)\, t^i,\qquad \sum_{i=1}^\infty g_i\, t^i\in \mathfrak g_{\mathbb K\langle\mathscr A\rangle}(\mathbb K),
\end{equation}
is a projection so that as a topological vector space $\mathfrak g_{\mathbb K\langle\mathscr A\rangle}(\mathbb K)$ is isomorphic to the direct sum of topological vector spaces $\mathfrak g_{\mathbb K\langle\mathscr A\rangle/\mathcal J_\mathbb K}(\mathbb K)\oplus \mathfrak g_{\mathcal H_{{\rm FdB}}^{cop}(\mathbb K}(\mathbb K)$. This implies that as a topological space $G_{\mathbb K\langle\mathscr A\rangle}(\mathbb K)$ is homeomorphic to the direct product of topological spaces $G_{\mathbb K\langle\mathscr A\rangle/\mathcal J_\mathbb K}(\mathbb K)\times G_\mathbb K[[r]]$.

\begin{R}
{\rm According to the Shirshov-Witt theorem (see, e.g., \cite{Re}), ${\rm Prim} \bigl(\mathbb K\langle\mathscr A\rangle/\mathcal J_\mathbb K\bigr)^*$ is a free Lie algebra. It is easily seen that the set of generators of ${\rm Prim} \bigl(\mathbb K\langle\mathscr A\rangle/\mathcal J_\mathbb K\bigr)^*$ is countable and each generator can be chosen to be homogeneous (i.e. belonging to some ${\rm Prim} \bigl(\mathbb K\langle\mathscr A\rangle/\mathcal J_\mathbb K\bigr)^*\cap \mathbb K_i\langle\mathbf X\rangle$). Let $\mathbf Y=\{Y_i :  Y_i\in \mathbb K_{n_i}\langle\mathbf X\rangle ,\ i\in\N\}$ be the set of generators of ${\rm Prim} \bigl(\mathbb K\langle\mathscr A\rangle/\mathcal J_\mathbb K\bigr)^*$ ordered such that $n_i\le n_j$ for $i\le j$ ($i,j\in\N$). Then the
bijective map $\mathbf X\rightarrow\mathbf Y$, $X_i\mapsto Y_i$, $i\in\N$, extends to an isomorphism of (non-graded) Hopf algebras $\varphi:\mathbb K\langle\mathbf X\rangle\rightarrow  \bigl(\mathbb K\langle\mathscr A\rangle/\mathcal J_\mathbb K\bigr)^*$ 
and induces an isomorphism between topological groups $G_{\mathbb K\langle\mathscr A\rangle}(\mathbb K)$ and $G_{\mathbb K\langle\mathscr A\rangle/\mathcal J_\mathbb K}(\mathbb K)$. (Thus the group of formal centers of equation \eqref{e1} is isomorphic to the group of characters of the shuffle Hopf algebra $\mathbb C\langle\mathscr A\rangle$.)
In turn, the transpose map $\varphi^*$ determines a Hopf algebra isomorphism between $\mathbb K\langle\mathscr A\rangle/\mathcal J_\mathbb K$ and the shuffle Hopf algebra $\mathbb K\langle\mathscr A\rangle$.  
}
\end{R}
\subsection{Equations with Finitely Many Terms}
The Hopf algebra approach to the center problem for equations \eqref{e1} with finitely many terms is similar to the one already described so we just sketch it leaving details to the readers.

Let $\mathscr A_N$ be the subalphabet of $\mathscr A$ consisting of letters $\alpha_1,\dots, \alpha_N$. Let $\mathbb K\langle \mathscr A_N\rangle$ be the graded Hopf subalgebra of $\mathbb K\langle \mathscr A\rangle$ generated by words in the letters of $\mathscr A_N$. We have the natural projection $(q_{N})_{\mathbb K}: \mathbb K\langle \mathscr A\rangle\rightarrow \mathbb K\langle \mathscr A_N\rangle$ sending each word in $\mathscr A^*$ containing letters of $\mathscr A\setminus\mathscr A_N$ to zero and identity on the other words. The dual of $\mathbb K\langle \mathscr A_N\rangle$ is the Hopf subalgebra $\mathbb K\langle\mathbf X_N\rangle$, $\mathbf X_N:=\{X_1,\dots, X_N\}$, of $\mathbb K\langle\mathbf X\rangle$ generated by polynomials in $I$ and variables in $\mathbf X_N$.  We have the natural projection
$(Q_N)_{\mathbb K}:\mathbb K\langle\mathbf X\rangle\rightarrow \mathbb K\langle\mathbf X_N\rangle$ defined by equating all variables $X_i$, $i>N$, to zero. The Hopf map $(Q_N)_{\mathbb K}$ extends by continuity to the projection $L(\mathbb K\langle \mathscr A\rangle;\mathbb K)\rightarrow L(\mathbb K\langle \mathscr A_N\rangle;\mathbb K)$ denoted by the same symbol. 

Next, the Lie algebra ${\rm Prim}\, \mathbb K\langle\mathscr A_N\rangle\subset {\rm Prim}\, \mathbb K\langle\mathscr A\rangle$ is the free Lie algebra generated by $X_1,\dots, X_N$ with the grading induced from ${\rm Prim}\, \mathbb K\langle\mathscr A\rangle$. The Lie algebra  $\mathfrak g_{\mathbb K\langle\mathscr A_N\rangle}(\mathbb K)$ of infinitesimal characters of $\mathbb K\langle\mathscr A_N\rangle$ is the closure of ${\rm Prim}\, \mathbb K\langle\mathscr A_N\rangle$ in $\mathfrak g_{\mathbb K\langle\mathscr A\rangle}(\mathbb K)$. Also, it coincides with  $(Q_N)_{\mathbb K}\bigl(\mathfrak g_{\mathbb K\langle\mathscr A\rangle}(\mathbb K)\bigr)$. Similarly, the Lie group $G_{\mathbb K\langle\mathscr A_N\rangle}(\mathbb K)$ of characters  of $\mathbb K\langle\mathscr A_N\rangle$ coincides with  $(Q_N)_{\mathbb K}\bigl(G_{\mathbb K\langle\mathscr A\rangle}(\mathbb K)\bigr)$. 

We set (see Section~5.1)
\[
(\rho_N)_{\mathbb K}:=\rho_\mathbb K|_{\mathbb K\langle\mathbf X_N\rangle}:\mathbb K\langle\mathbf X_N\rangle\rightarrow \mathcal H_{{\rm FdB}}^{cop}(\mathbb K)^*.
\]
Then the transpose map $(\rho_N)_{\mathbb K}^*: \mathcal H_{{\rm FdB}}^{cop}(\mathbb K)\rightarrow\mathbb K\langle\mathscr A_N\rangle$ satisfies
\[
(\rho_N)_{\mathbb K}^*=(q_N)_{\mathbb K}\circ\rho_{\mathbb K}^*.
\]
In particular, setting $\mathcal P_i^N:=\mathcal P_i$ for $i\le N$ and
$\mathcal P_i^N(X_1,\dots, X_N):=\mathcal P_i(X_1,\dots, X_N,0,0,\dots)$  for $i>N$, we obtain
\begin{Proposition}\label{pr5.6}
Map $(\rho_N)_\mathbb K^*$ is a monomorphism of graded Hopf algebras given on the generators by the formulas
\begin{equation}\label{eq6.1}
(\rho_N)_\mathbb K^*(t_i)=(q_N)_\mathbb K\bigl(\mathcal P_i(\alpha_1,\dots, \alpha_i)\bigr)=:\mathcal P_i^N(\alpha_1,\dots,\alpha_N),\qquad i\in\N.
\end{equation}
\end{Proposition}
Substituting $X_i=0$ for all $i>N$ in Theorem \ref{recur} we get the following recurrence relations (originally established in \cite{D} by a different method):
\begin{equation}\label{eq5.12}
\mathcal P_n^N(X_1,\dots,X_N)=\sum_{i=1}^{\min(n,N)}(n-i+1)\cdot \mathcal P_{n-i}^N(X_1,\dots, X_N)\cdot X_i,\qquad n\in\N.
\end{equation}

Let $\mathscr X_N\subset\mathscr X$ be the subsemigroup of sequences $a=(a_1,a_2,\dots)$ with $a_i=0$ for $i>N$. Each $a\in \mathscr X_N$ determines equation
\begin{equation}\label{eq5.13}
\frac{dv}{dx}=\sum_{i=1}^{N}a_{i}(x)v^{i+1},\qquad x\in I_{T}.
\end{equation}
In turn, map $E:\mathscr X\rightarrow L(\mathbb C\langle \mathscr A\rangle;\mathbb C)$, see \eqref{e2.16}, sends $\mathscr X_N$ into the group of characters $G_{\mathbb C\langle \mathscr A_N\rangle}(\mathbb C)$ of $\mathbb C\langle \mathscr A_N\rangle$. The subgroup $\widehat{\mathscr C}^N:=E({\mathscr X_N}\cap{\mathscr C})\subset\widehat{\mathscr C}$ is called the group of centers of equation \eqref{eq5.13}. Its closure in $G_{\mathbb C\langle \mathscr A_N\rangle}(\mathbb C)$ denoted by $\widehat{\mathscr C}_f^N$ is called the group of formal centers of this equation.
The Lie algebra of $\widehat{\mathscr C}_{f}^N$ is then the Lie algebra $\mathfrak g_{\mathbb C\langle\mathscr A_N\rangle/\mathcal J^N_\mathbb C}(\mathbb C)$  of infinitesimal characters of the quotient Hopf algebra
$\mathbb C\langle\mathscr A_N\rangle/\mathcal J^N_\mathbb C$, where $\mathcal J^N_\mathbb C$ is the Hopf ideal generated by the displacement polynomials $\mathcal P_i^N$, $i\in \N$. One easily checks that $\mathcal J^N_\mathbb C=(q_N)_\mathbb C (\mathcal J_\mathbb C)$
and that $\widehat{\mathscr C}_{f}^N$ is the group of characters of $\mathbb C\langle\mathscr A_N\rangle/\mathcal J^N_\mathbb C$.

Among other results, let us mention that group $G_{\mathbb C\langle\mathscr A\rangle}(\mathbb C)$ is isomorphic to the semidirect product of group $G_{\mathbb C\langle\mathscr A_2\rangle}(\mathbb C)$ (corresponding to the Abel differential equation) and a certain normal subgroup of $\widehat{\mathscr C}_f$ (see \cite[Prop.\,3.9]{Br2}) and that $G_{\mathbb C\langle\mathscr A_2\rangle}(\mathbb C)$ as a topological space is homeomorphic to the product of topological spaces $\widehat{\mathscr C}_{f}^{\,2}\times G_\mathbb C[[r]]$ (see \cite[Prop.\,3.12]{Br2}).

An important class of equations \eqref{e1} with finitely many terms is determined by the subsemigroup of rectangular paths
 $\mathscr X_{rect}\subset \mathscr X$ of elements $a\in \mathscr X$ whose first integrals 
$\widetilde a: I_{T}\rightarrow\mathbb C^{\infty}$ are paths consisting of segments each going in the direction of some particular coordinate. Every center determined by $a\in \mathscr X_{rect}$ is universal (i.e. $\mathscr C\cap\mathscr X_{rect}\subset\mathscr U$). This fact was established in \cite{Br3} by means of
the deep result of Cohen \cite{Co}. 
Note that $G(\mathscr X_{rect}):=E(\mathscr X_{rect})\subset G_{\mathbb C\langle\mathscr A\rangle}(\mathbb C)$ is a dense subgroup generated by elements $e^{c_{n}X_{n}t^{n}}$, $c_{n}\in\mathbb C$, $n\in\N$. (In particular, $G(\mathscr X_{rect})$ is isomorphic to the free product of countably many copies of $\mathbb C$.) Thus we have  $G(\mathscr X_{rect})\cap\widehat{\mathscr C}_f=\{I\}$.

\sect{Generalized Displacement Polynomials}
In this section we reveal the algebraic nature of the generalized displacement polynomials, prove recurrence relations for them and show that, when considered in the shuffle Hopf algebra $\mathbb K\langle\mathscr A\rangle$, their values in $t$ can be computed by means of the displacement polynomials $\mathcal P_i$. Also, we compute the values $\varepsilon_\mathbb K(\widetilde{\mathcal P}_i)$, $i\in\N$, of the augmentation homomorphism $\varepsilon_\mathbb K:\mathbb K\langle{\mathbf X}\rangle \rightarrow\mathbb K$.
\subsection{Recurrence Relations}
Let $\mathbb K[[z]]$ be the algebra of formal power series in variable $z$ with coefficients in the field of characteristic zero $\mathbb K$. By $D,\ \!L:\mathbb K[[z]]\rightarrow\mathbb K[[z]]$ we denote the differentiation and the left translation operators defined on $f(z)=\sum_{k=0}^{\infty}c_{k}z^{k}$ by
\begin{equation}\label{e6.1}
(Df)(z):=\sum_{k=0}^{\infty}(k+1)c_{k+1}z^{k}\qquad {\rm and}\qquad (Lf)(z):=\sum_{k=0}^{\infty}c_{k+1}z^{k}.
\end{equation}
Let ${\mathcal A}_\mathbb K(D,L)$ be the associative algebra with unit $I:={\rm Id}|_{\mathbb K[[z]]}$ of polynomials with coefficients in $\mathbb  K$ in variables $I$, $D$ and $L$. We define
$[A,B]:=AB-BA$ for $A,B\in  {\mathcal A}_\mathbb K(D,L)$.
Then we have (see, e.g., \cite[Lm.\,3.2]{Br2})
\[
[DL^{i},DL^{j}]=(i-j)DL^{i+j+1}\qquad i,j\in\mathbb Z_+.
\]
In particular, the vector space $\mathfrak L_\mathbb K\subset {\mathcal A}_\mathbb K(D,L)$ over $\mathbb K$ generated by $DL^i$, $i\in\mathbb Z_+$, and equipped with the above bracket is a Lie algebra.
It is easily seen that the linear map $\sigma_\mathbb K:  {\rm Prim}\,\mathcal H_{{\rm FdB}}^{cop}(\mathbb K)^*\rightarrow\mathfrak L_\mathbb K$ defined on the generators by $\sigma(t_i'):=DL^{i-1}$, $i\in\N$, is an isomorphism of Lie algebras. By the universal property of enveloping algebras it extends to the morphism of associative  algebras $\sigma_\mathbb K:\mathcal H_{{\rm FdB}}^{cop}(\mathbb K)^*\rightarrow {\mathcal A}_\mathbb K(D,L)$. 
 The following result exposes the algebraic nature of the generalized displacement polynomials.
\begin{Proposition}\label{prop6.1}
Element 
\[
(\sigma_\mathbb K\otimes {\rm Id})(\mathscr P_\mathbb K):=(1\otimes I-\sum_{j=1}^\infty DL^{j-1}\otimes X_j\, t^j)^{-1}\in {\mathcal A}_\mathbb K(D,L)\otimes L(\mathbb K\langle\mathscr A\rangle;\mathbb K),
\]
see \eqref{gf}, satisfies
\begin{equation}\label{eq6.2}
(\sigma_\mathbb K\otimes {\rm Id})(\mathscr P_\mathbb K)(z^m)=\sum_{i=0}^m \widetilde{\mathcal P}_i(X_1,\dots, X_i;m)\cdot z^{m-i}\cdot t^i\quad {\rm for\ all}\quad m\in\mathbb Z_+.
\end{equation}
\end{Proposition}
\begin{proof}
By definition,
\[
\begin{array}{l}
\displaystyle
(\sigma_\mathbb K\otimes {\rm Id})(\mathscr P_\mathbb K)(z^m)=I(z^m)+\sum_{i=1}^\infty\left(\sum_{j=1}^\infty DL^{j-1}\otimes X_j\, t^j\right)^i(z^m)\medskip\\
\displaystyle =z^m+\sum_{i=1}^\infty\left(\sum_{i_1+\cdots+i_k=i}(DL^{i_k-1}\cdots  DL^{i_1-1})(z^m)\otimes X_{i_k}\cdots  X_{i_1}\right) t^i \medskip\\
\displaystyle =z^m+\sum_{i=1}^m\left(\sum_{i_1+\cdots+i_k=i}(m-i_1+1)(m-i_1-i_2+1)\cdots (m-i+1)\, z^{m-i}\,X_{i_k}\cdots  X_{i_1}\right) t^i\medskip\\
\displaystyle :=\sum_{i=0}^m \widetilde{\mathcal P}_i(X_1,\dots, X_i;m)\cdot z^{m-i}\cdot t^i.
\end{array}
\]
\end{proof}
Using this proposition we establish two basic recurrence relations for the generalized displacement polynomials.
\begin{Th}\label{te6.2}
For all $n\in\N$, $t\in\mathbb C$,
\begin{equation}\label{e6.3}
\widetilde{\mathcal P}_n(X_1,\dots, X_n;t)=\sum_{j=1}^n \widetilde{\mathcal P}_{n-j}(X_1,\dots, X_{n-j}; t-j)\cdot X_j
\end{equation}
and
\begin{equation}\label{e6.4}
\widetilde{\mathcal P}_n(X_1,\dots, X_n;t)=(t-n+1)\cdot\sum_{j=1}^n X_j\cdot \widetilde{\mathcal P}_{n-j}(X_1,\dots, X_{n-j};t).
\end{equation}
\end{Th}
\begin{proof}
We set for brevity $\widetilde{\mathscr P}:=(\sigma_\mathbb C\otimes {\rm Id})(\mathscr P_\mathbb C)$.
Applying to the identity
\[
\mathscr P_\mathbb C=1\otimes I+\mathscr P_\mathbb C\times\sum_{j=1}^\infty t_j'\otimes X_j\,t^j
\]
operator $\sigma_\mathbb C\otimes {\rm Id}$ and using Proposition \ref{prop6.1} we obtain for all $m\in\mathbb Z_+$ (below $\times$ stands for the product on ${\mathcal A}_\mathbb C(D,L)\otimes L(\mathbb C\langle\mathscr A\rangle;\mathbb C)$),
\[
\begin{array}{l}
\displaystyle
\sum_{i=0}^m \widetilde{\mathcal P}_i(X_1,\dots, X_i;m)\cdot z^{m-i}\cdot t^i=\widetilde{\mathscr P}(z^m)=\left(1\otimes I+\widetilde{\mathscr P}\times\sum_{j=1}^\infty DL^{j-1}\otimes X_j\,t^j\right)(z^m)\medskip
\\
\displaystyle =z^m+\sum_{j=1}^m\widetilde{\mathscr P}(z^{m-j})\cdot (m-j+1)\cdot X_j\, t^j\medskip
\\
\displaystyle =z^m+\sum_{j=1}^m (m-j+1)
\left(\sum_{k=0}^{m-j} \widetilde{\mathcal P}_{k}(X_1,\dots, X_{k};m-j)\, z^{m-j-k}\cdot r^k\right)\cdot X_j\, t^j\medskip
\\
\displaystyle =z^m+\sum_{i=1}^m\left(\sum_{j=1}^{i}  \widetilde{\mathcal P}_{i-j}(X_1,\dots, X_{i-j};m-j)\cdot X_j \right)  z^{m-i}\cdot t^i.
\end{array}
\]
Choosing here $m\ge n$ and equating coefficients of $t^n$ in both sides of the above equation we get
\[
\widetilde{\mathcal P}_n(X_1,\dots, X_n;m)=\sum_{j=1}^n (m-j+1)\cdot \widetilde{\mathcal P}_{n-j}(X_1,\dots, X_{n-j};m-j)\cdot X_j.
\]
Since coefficients of $\widetilde{\mathcal P}_1,\dots, \widetilde{\mathcal P}_n$ are polynomials in $m$ and the above identity is valid for all sufficiently large natural $m$, it is valid for all complex numbers $m$ as well. This completes the proof of \eqref{e6.3}.

To prove the second relation we use a different presentation of $\mathscr P_\mathbb C$:
\[
\mathscr P_\mathbb C=1\otimes I+\left(\sum_{j=1}^\infty t_j'\otimes X_j\,t^j \right)\times\mathscr P_\mathbb C.
\]
As before, applying to this equation operator $\sigma_\mathbb C\otimes {\rm Id}$ and using Proposition \ref{prop6.1} we obtain for all $m\in\mathbb Z_+$,
\[
\begin{array}{l}
\displaystyle \sum_{i=0}^m \widetilde{\mathcal P}_i(X_1,\dots, X_i;m)\cdot z^{m-i}\cdot t^i=\widetilde{\mathscr P}(z^m)=\left(1\otimes I+\left(\sum_{j=1}^\infty DL^{j-1}\otimes X_j\,t^j\right)\times \widetilde{\mathscr P}\right)(z^m)\medskip
\\
\displaystyle=z^m+\sum_{j=1}^\infty X_j\cdot DL^{j-1}(\widetilde{\mathscr P}(z^m))\, t^j\medskip\\
\displaystyle =z^m+\sum_{j=1}^\infty X_j\cdot\left( \sum_{k=0}^m DL^{j-1}(z^{m-k})\cdot \widetilde{\mathcal P}_k(X_1,\dots, X_k; m)\, t^k\right) t^{j}
\medskip\\
\displaystyle = z^m+\sum_{j=1}^{m} X_j\cdot\left( \sum_{k=0}^{m-j} (m-j-k+1)\cdot \widetilde{\mathcal P}_k(X_1,\dots, X_k; m)\, z^{m-j-k}\cdot t^k\right) t^{j}
\medskip\\
\displaystyle =z^m+\sum_{i=1}^m (m-i+1)\left(\sum_{j=1}^i X_j\cdot \widetilde{\mathcal P}_{i-j}(X_1,\dots, X_{i-j};m)\right) z^{m-i}\cdot t^i.
\end{array}
\]
Choosing here $m\ge n$ and equating coefficients of $t^n$ in both sides of the above equation we get
\[
\widetilde{\mathcal P}_n(X_1,\dots, X_n;m)=(m-n+1)\cdot\sum_{j=1}^n X_j\cdot \widetilde{\mathcal P}_{n-j}(X_1,\dots, X_{n-j};m).
\]
Again, since coefficients of $\widetilde{\mathcal P}_1,\dots, \widetilde{\mathcal P}_n$ are polynomials in $m$ and the above identity is valid for all sufficiently large natural $m$, it is valid for all complex numbers $m$ as well. This completes the proof of \eqref{e6.4} and of the theorem.
\end{proof}
\subsection{Generalized Displacement Polynomials in the Shuffle Hopf Algebra}
In this section we study polynomials $\widetilde {\mathcal P}_i(\alpha_1,\dots,\alpha_i;t)\in (\mathbb K\langle\mathscr A\rangle,\shuffle)$, $i\in \mathbb Z_+$, where $\widetilde{\mathcal P}_0:=1$ (as before, the product of words in $\mathscr A^*$ is defined by concatenation).  
\subsubsection{{\bf Expression via the Bell Polynomials}}
The main result of this section shows that the generalized displacement polynomials belong to the image of the natural monomorphism $\rho_\mathbb K^*:\mathcal H_{{\rm FdB}}^{cop}(\mathbb K)\rightarrow\mathbb K\langle \mathscr A\rangle$, see Section~5.1.1. To formulate the result, we introduce
polynomials $\mathcal B_k\in\mathbb Q[t_1,\dots, t_k,t]$, $k\in\N$, 
\begin{equation}\label{genbell1}
\mathcal B_k(t_1,\dots, t_k,t):=  \sum_{l_1+2l_2+\cdots +kl_k=k}\left(\prod_{l=1-l_1-\dots -l_k}^{0} (t+l)\right)\cdot\frac{t_1^{l_1}\cdots t_k^{l_k}}{l_1!\cdots l_k!}.
\end{equation}
Also, we set $\mathcal B_0:=1$.
\begin{Proposition}\label{prop6.3}
For all $i,j\in\mathbb Z_+$, $i-j\ge 1$ polynomials \eqref{genbell1} satisfy
\begin{equation}\label{genbell2}
\mathcal B_{i-j}(t_1,\dots, t_{i-j},j+1)=\frac{(j+1)!}{(i+1)!}B_{i+1,j+1}(1,2!t_1,3!t_2,\dots, (i-j+1)! t_{i-j}).
\end{equation}
\end{Proposition}
\begin{proof}
From \eqref{genbell1} with $k:=i-j$, $t=j+1$ we get, cf. \eqref{e2.19},
\[
\begin{array}{c}
\displaystyle
\mathcal B_{i-j}(t_1,\dots, t_{i-j},j+1)=\sum_{l_1+2l_2+\cdots +(i-j)l_{i-j}=i-j}\left(\prod_{l=1-l_1-\dots -l_{i-j}}^{0} (j+1+l)\right)\cdot\frac{t_1^{l_1}\cdots t_{i-j}^{l_{i-j}}}{l_1!\cdots l_{i-j}!}\medskip\\
\displaystyle \qquad \ =\sum_{\stackrel{l_0+l_1+\cdots +l_{i-j}=j+1}{_{l_1+2l_2+\cdots +(i-j)l_{i-j}=i-j}}}\frac{(j+1)!}{l_0!\, l_1!\cdots l_{i-j}!}t_1^{l_1}\cdots t_{i-j}^{l_{i-j}}\medskip\\
\displaystyle \qquad\qquad\quad =
\frac{(j+1)!}{(i+1)!}B_{i+1,j+1}(1,2!t_1,3!t_2,\dots, (i-j+1)! t_{i-j}),
\end{array}
\]
as required.
\end{proof}
We are ready to formulate the main result of this section. 
\begin{Th}\label{gpshuffle}
The following identities hold in  $(\mathbb K\langle\mathscr A\rangle,\shuffle)$ for all $i\in\Z_+$, $t\in\mathbb K$:
\begin{equation}\label{equ6.5}
\widetilde{\mathcal P}_i(\alpha_1,\dots,\alpha_i;t)=\mathcal B_i(\mathcal P_1(\alpha_1),\dots, \mathcal P_{i}(\alpha_1,\dots,\alpha_i), t-i+1).
\end{equation}
\end{Th}
In particular, 
\begin{equation}\label{equ6.6}
\widetilde {\mathcal P}_i(\alpha_1,\dots,\alpha_i;t)\in \rho_\mathbb K^*\bigl(\mathcal H_{{\rm FdB}}^{cop}(\mathbb K)\bigr)\cap\mathbb K_i\langle\mathscr A\rangle\quad {\rm for\ all}\quad t\in\mathbb K,\ i\in\mathbb Z_+.
\end{equation}
\begin{proof}
Let us consider algebra ${\mathcal A}_\mathbb C(D,L)$ with the grading defined on the generators by ${\rm deg}(D)={\rm deg}(L)=1$, ${\rm deg}(I)=0$. The closure of ${\mathcal A}_\mathbb C(D,L)$ in the adic topology induced by the grading is naturally identified with the subalgebra $\mathcal A_*$ of the algebra ${\mathcal A}_\mathbb C(D,L)[[t]]$ of formal power series in $t$ with coefficients in ${\mathcal A}_\mathbb C(D,L)$ of series of the form
\[
f(t)=\sum_{i=0}^\infty f_i\, t^i,\qquad f_i\in\mathcal A_\mathbb C(D,L),\quad {\rm deg}(f_i)=i,\quad i\in\mathbb Z_+.
\]
Map $\sigma_\mathbb C$ (see Section~6.1) extends by continuity to the morphism of topological algebras $L(\mathcal H_{{\rm FdB}}^{cop}(\mathbb C);\mathbb C)\rightarrow \mathcal A_*$ (denoted by $\sigma_\mathbb C$ as well) which maps the Lie algebra $\mathfrak g_{\mathcal H_{{\rm FdB}}^{cop}(\mathbb C)}(\mathbb C)$ of infinitesimal characters of $\mathcal H_{{\rm FdB}}^{cop}(\mathbb C)$ isomorphically onto the closure ${\rm cl}(\mathfrak L_\mathbb C)$ of the Lie algebra $\mathfrak L_\mathbb C={\rm span}_\mathbb C\{DL^{i-1},  i\in\N\}$ in $\mathcal A_*$. The latter consists of series of the form
\[
g(t)=\sum_{i=1}^\infty d_i\, DL^{i-1}\, t^i,\qquad d_i\in\mathbb C,\  i\in\N,
\]
with the bracket extending by continuity the bracket on $\mathfrak L_\mathbb C$. In particular, $\sigma_\mathbb C$ maps the group $G_{\mathcal H_{{\rm FdB}}^{cop}(\mathbb C)}(\mathbb C)$ of characters of 
$\mathcal H_{{\rm FdB}}^{cop}(\mathbb C)$ isomorphically onto the group $\exp({\rm cl}(\mathfrak L_\mathbb C))\subset\mathcal A_*$. 
\noindent 
Further, $\sigma_\mathbb C$ transfers ordinary differential equation \eqref{e3.3} with values in $L(\mathcal H_{{\rm FdB}}^{cop}(\mathbb C);\mathbb C)$ corresponding to $a\in\mathscr X$ to equation
\begin{equation}\label{eq6.5}
\frac{dF}{dx}=\left(\,\sum_{i=1}^{\infty} a_{i}(x)\, DL^{i-1}\, t^{i}\right) F,\ \ \ x\in I_{T}.
\end{equation}
The fundamental solution $H_a$ of \eqref{e3.3} then goes to the fundamental solution $F_a:=\sigma_\mathbb C\circ H_a$ of \eqref{eq6.5} so that its value at $x=T$ is given by the formula 
\begin{equation}\label{eq6.6}
F_a(T;t)=I+\sum_{i=1}^{\infty}\left(\ \!\sum_{i_{1}+\cdots +i_{k}=i}I_{i_{1},\dots, i_{k}}(a)\,DL^{i_{k}-1}\cdots DL^{i_{1}-1}\right)t^{i}.
\end{equation}
Given $a\in\mathscr X$ the expression in the brackets of \eqref{eq6.6} is a linear operator, in what follows denoted by $I_i(a)$, on the algebra of complex polynomials $\mathbb C[z]$. The next result computes its values $I_i(a)(z^m)$ on the basis $\{z^m\}_{m\in\mathbb Z_+}$ of $\mathbb C[z]$.
\begin{Proposition}\label{prop6.4}
(1) If $m=0,\dots, i-1$, then $I_i(a)(z^m)=0$. 

(2) If 
$m\ge i$, then
\[
\begin{array}{l}
\displaystyle
I_i(a)(z^m)=\left(\sum_{i_{1}+\cdots +i_{k}=i}p_{i_1,\dots, i_k}(m)\cdot I_{i_{1},\dots, i_{k}}(a)\right) z^{m-i}\medskip\\
\displaystyle \qquad\qquad\ =\frac{(m-i+1)!}{(m+1)!} B_{m+1,m-i+1}(1, 2!p_1(a),\dots, (i+1)!p_{i}(a))\cdot z^{m-i},
\end{array}
\]
where $p_k(a)$ is the coefficient of $r^{k+1}$ in the series expansion of the first return map $\mathcal P(a)$, see \eqref{e4}.
\end{Proposition}
\begin{proof}
(1) This follows directly from the definition of $I_i(a)$.\smallskip

\noindent (2) Applying $F_a$ to $s(z, r):=r\, (1-rz)^{-1}=\sum_{m=0}^\infty r^{m+1}\cdot z^m\in \mathbb C[[z]]$ we obtain, see, e.g., \cite[Sect.\,4]{Br6}, 
\begin{equation}\label{equ6.9}
F_a(T;t)\bigl(s(z,r)\bigr)=s\bigl(z, \mathcal P(a)(r\, t)/t\bigr)
\in \mathbb C[[t,z,r]].
\end{equation}
For $t$ and $z$ fixed, the latter is the composition of series  $s(z;r)\in\mathbb C[[r]]$ and 
$\frac{\mathcal P(a)(r \, t)}{t}\in\mathbb C[[r]]$. Hence, using the Fa\`{a} di Bruno formula, cf. \eqref{e2.18}, we obtain
\[
\begin{array}{l}
\displaystyle
s(z,r)+\sum_{i=1}^\infty I_i(a)(s(z,r))t^i=:F_a(T;t)\bigl(s(z,r)\bigr)\medskip\\
\displaystyle =r+\sum_{m=1}^\infty\left(\sum_{k=0}^m\frac{(k+1)!}{(m+1)!}z^k t^{m-k} B_{m+1,k+1}(1, 2!p_1(a),\dots, (m-k+1)!p_{m-k}(a))\right) r^{m+1}.
\end{array}
\]
Equating coefficients of $t^i$ in both sides of this identity we get
\[
I_i(a)(s(z,r))=\sum_{m=i}^\infty\left(\frac{(m-i+1)!}{(m+1)!} B_{m+1,m-i+1}(1, 2!p_1(a),\dots, (i+1)!p_{i}(a))\right) z^{m-i}\cdot r^{m+1}.
\]
On the other hand, the straightforward application of operators $DL^{i_{k}-1}\cdots DL^{i_{1}-1}$ to the basis $\{z^m\}_{m\in\mathbb Z_+}$ of $\mathbb C[z]$ shows that the left-hand side of the above expression equals
\[
\sum_{m=0}^\infty I_i(a)(z^m)\, r^{m+1}=
\sum_{m=i}^\infty\left(\sum_{i_{1}+\cdots +i_{k}=i}p_{i_1,\dots, i_k}(m)\cdot I_{i_{1},\dots, i_{k}}(a)\right) z^{m-i}\cdot r^{m+1}.
\]
Comparing coefficients of $r^{m+1}$ in the latter and the former equations we get the required identity.
\end{proof}
Further, recall that there is an isomorphism between the algebra $\mathcal I_\mathbb C(\mathscr X)$ generated by iterated integrals on $\mathscr X$ and the shuffle algebra $(\mathbb C\langle\mathscr A\rangle, \shuffle)$ sending functions $I_{i_1,\dots, i_k}$ on $\mathscr X$ to words $\alpha_{i_k}\dots\alpha_{i_1}\in\mathscr A^*$, $i_1,\dots, i_k, k\in\N$ (see Section~3.3). Applying this algebra isomorphism to the coefficients of $z^{m-i}$ in the identity of Proposition \ref{prop6.4} and noting that their images belong to $\mathbb Q\langle\mathscr A\rangle\, (\subset \mathbb K\langle\mathscr A\rangle)$, from Proposition \ref{prop6.3} we obtain that equation \eqref{equ6.5} is valid for all integers $t\ge i$, $i\in \N$. Since in both sides of \eqref{equ6.5} are polynomials in $t$ and this equation is valid for infinitely many values of $t$, it is valid for all $t\in\mathbb K$. This completes the proof of Theorem \ref{gpshuffle} for $i\in\N$. For $i=0$ the statement of the theorem is trivially true.
\end{proof}
\begin{R}\label{rem6.6}
{\rm Note that representation $\sigma_\mathbb C\circ\Theta^{-1}:(G_\mathbb C[[r]],\circ)\rightarrow \exp({\rm cl}(\mathfrak L_\mathbb C))\subset\mathcal A_*$, where
$\Theta^{-1}: (G_\mathbb C[[r]],\circ)\rightarrow
(G_{\mathcal H_{{\rm FdB}}^{cop}(\mathbb C)},*)$ is the isomorphism given by \eqref{eq4.7}, is faithful and due to Proposition \ref{prop6.4} coincides with the well-known `smallest faithful representation' of $(G_\mathbb C[[r]],\circ)$, see, e.g.,
\cite[Sect.\,3.1]{FM}; here $\exp({\rm cl}(\mathfrak L_\mathbb C))$ is considered as the subgroup of the group of invertible linear operators acting on the space of complex polynomials $\mathbb C[z]$. Thus, the values of the generalized displacement polynomials for $t\in\N$ are the `matrix entries' (written with respect to the basis $\{z^m\}_{m\in\mathbb Z_+}$ of $\mathbb C[z]$)
of the composite homomorphism $F_{\cdot}=(\sigma_\mathbb C\circ\Theta^{-1})\circ\mathcal P:(\mathscr X,*)\rightarrow \exp({\rm cl}(\mathfrak L_\mathbb C))$.
}
\end{R}
For $a=(a_1,a_2,\dots)\in\mathscr X$, we set
\[
\begin{array}{c}
\widetilde{\mathcal P}_i\bigl(\int\!a_1,\dots, \int\!a_i; t\bigr):= \displaystyle \sum_{i_{1}+\cdots +i_{k}=i}p_{i_1,\dots, i_k}(t)\cdot I_{i_{1},\dots, i_{k}}(a).
\end{array}
\]
Note that if we introduce bounded linear operators $E_{a_i}:L^{\infty}(I_T)\rightarrow L^\infty (I_T)$,
\[
(E_{a_i}f)(x):=\int_0^x a_i(s)f(s)\, ds,\quad x\in I_T,\quad i\in\N,
\]
then for each $t\in\mathbb C$,
\[
\begin{array}{l}
\widetilde{\mathcal P}_i\bigl(\int\!a_1,\dots, \int\!a_i; t\bigr)=\bigl(\widetilde{\mathcal P}_i(E_{a_{1}},\dots, E_{a_{i}};t)(1)\bigr)(T);
\end{array}
\]
here the product of operators is given by composition.\smallskip

As the corollary of Theorem \ref{gpshuffle} and Proposition \ref{prop6.4} we get (cf. Section~3.3):
\begin{C}\label{cor6.5}
(1) Suppose $a_k=(a_{1k},a_{2k},\dots)\in\mathscr X$, $k=1,2$. Then
\begin{equation}\label{eq6.10}
\begin{array}{c}
\widetilde{\mathcal P}_i\bigl(\int a_{11},\dots, \int a_{i1}; t\bigr)=\widetilde{\mathcal P}_i\bigl(\int a_{12},\dots, \int a_{i2}; t\bigr)\quad {\rm for\ all}\quad t\in\mathbb C,\ i\in\N,
\end{array}
\end{equation}
if and only if $a_1*a_2^{-1}\in\mathscr C$.\smallskip

\noindent (2) Let $g\in G_{\mathbb C\langle\mathscr A\rangle}(\mathbb C)$, $h\in \widehat{\mathscr C}_f$ be some
characters of the shuffle Hopf algebra $\mathbb C\langle\mathscr A\rangle$. Then
\begin{equation}\label{eq6.11}
(g*h)\left(\widetilde{\mathcal P}_i(\alpha_1,\dots,\alpha_i;t)\right)=g\left(\widetilde{\mathcal P}_i(\alpha_1,\dots,\alpha_i;t)\right)\quad {\rm for\ all}\quad t\in\mathbb C,\ i\in\N.
\end{equation}
\end{C}
\begin{proof}
(1) Since the first return map of equation \eqref{e1} $\mathcal P$ satisfies $\mathcal P(a_1)=\mathcal P(a_2)$ $\Leftrightarrow$ $a_1*a_2^{-1}\in\mathscr C$, the result follows from
Theorem \ref{gpshuffle}.
\smallskip

\noindent (2) Since, by Theorem \ref{gpshuffle}, $\widetilde{\mathcal P}_i(\alpha_1,\dots,\alpha_i;t)\in
\rho_\mathbb C^*\bigl(\mathcal H_{{\rm FdB}}^{cop}(\mathbb C)\bigr)$ each $h\in \widehat{\mathscr C}_f$ vanishes on it, see Section~5.3. This implies  \eqref{eq6.11}.
\end{proof}
\begin{R}
{\rm Suppose  $u_1,u_2,\dots $ is a sequence of complex-valued Lipschitz functions on $I_T$ such that $u_i(0)=0$ for all $i$. We define the sequence
$a_1,a_2,\dots$ of functions in $L^\infty(I_T)$ from the equality of formal power series
\begin{equation}\label{rem6.14}
\sum_{i=1}^\infty a_i(x)\,t^{i+1}=\frac{\sum_{k=1}^\infty u_k'(T-x)\,t^{k+1}}{1+\sum_{k=1}^\infty(k+1)u_k(T-x)\,t^k},\qquad x\in I_T.
\end{equation}
The differential equation
\[
\frac{dv}{dx}=\sum_{i=1}^\infty a_i(x)\,v^{i+1}
\]
has a solution $v(x;r)$, $x\in I_T$, $v(0;r)=r$, presenting as a formal power series in $r$ with Lipschitz coefficients. Identity \eqref{rem6.14} implies that
\[
v(x;r)+\sum_{i=1}^\infty u_i(T-x)\,v(x;r)^{i+1}=r+\sum_{i=1}^\infty u_i(T) r^{i+1}\qquad {\rm for\ all}\quad x\in I_T,
\]
cf. the second part of the proof of \cite[Th.\,3.7]{Br6} for similar arguments. Since the first return map of the above equation is given by formula \eqref{e4} as well, 
\begin{equation}\label{rem6.15}
\begin{array}{l}
\mathcal P_i\bigl(\int\!a_1,\dots, \int\!a_i\bigr)=u_i(T)\qquad {\rm for\ all}\quad i\in\N.
\end{array}
\end{equation}
Here $\mathcal P_i\bigl(\int\!a_1,\dots, \int\!a_i\bigr):=\widetilde{\mathcal P}_i\bigl(\int\!a_1,\dots, \int\!a_i;i\bigr)$. 

For instance, assume that $u_i(T):=\frac{t_i}{(i+1)!}$ for some $t_i\in\mathbb C$, $i\in\N$. Then by Theorem \ref{gpshuffle} and \eqref{rem6.15} we obtain for all natural numbers $i,j$ such that $i-j\ge 1$,
\begin{equation}\label{equ6.16}
\begin{array}{l}
\widetilde{\mathcal P}_{i-j}\bigl(\int\!a_1,\dots, \int\!a_{i-j};i-1\bigr)=\displaystyle \frac{j!}{i!}B_{i,j}(1,t_1,\dots, t_{i-j}).
\end{array}
\end{equation}
Thus the Bell polynomials can be expressed in many different ways via the values of the generalized displacement polynomials; e.g., one can choose $u_i(x)=\frac{t_i x^l}{T^l(i+1)!}$, $l, i\in\N$. As an illustration,
let us consider in details the case of $u_i(x)=\frac{t_i x}{T(i+1)!}$, $i\in\N$. Then \eqref{rem6.14} becomes
\[
\begin{array}{l}
\displaystyle
\sum_{i=1}^\infty a_i(x)\,t^{i+1}=\frac{\sum_{k=1}^\infty \frac{t_k }{T(k+1)!}\,t^{k+1}}{1+\sum_{k=1}^\infty \frac{T-x}{T}\cdot\frac{t_k}{k!}\,t^k}\medskip\\
\displaystyle
\qquad\qquad\qquad =\!\left(\sum_{k=1}^\infty \frac{t_k }{T(k+1)!}\,t^{k+1}\right)\!\cdot\!\left(\sum_{l=0}^\infty \left(\frac{x-T}{T}\right)^l\!\cdot\!\left[\sum_{k=1}^\infty \frac{t_k}{k!}\,t^k\right]^l\right)\medskip\\
\displaystyle \qquad\qquad\qquad =\!\left(\sum_{k=1}^\infty \frac{t_k }{T(k+1)!}\,t^{k+1}\right)\!\cdot\!\left(\sum_{l=0}^\infty \left(\frac{x-T}{T}\right)^l\!\cdot\! \left[l!\!\cdot\!\sum_{m=l}^\infty B_{m,l}(t_1,\dots,t_{m-l+1})\frac{t^m}{m!}\right]\right)\!.
\end{array}
\]
(Equality of expressions in the square brackets in lines two and three above follows from the Fa\`{a} di Bruno formula, see, e.g.,  \cite[page\,133]{Com}; here $B_{0,0}:=1$, $B_{m,0}:=0$, $m\ge1$.)

The latter implies that \eqref{equ6.16} is valid for
\[
a_i(x)=\frac{1}{(i+1)!}\sum_{k=1}^i \frac{t_k }{T}\cdot\binom{i+1}{i-k}\cdot\left(\sum_{l=0}^{i-k} l ! \cdot B_{i-k,l}(t_1,\dots,t_{i-k-l+1})\cdot\left(\frac{x-T}{T}\right)^l\right),\quad i\in\N.
\]
}
\end{R}


\subsubsection{{\bf Recurrence Relations}}
We apply Theorem \ref{gpshuffle} to establish several `recurrence relations' for polynomials $\widetilde{\mathcal P}_i(\alpha_1,\dots,\alpha_i;\cdot)$, $i\in\N$. Their proofs are based on known recurrence relations for the Bell polynomials, see, e.g., \cite{Com, Rio, Rom, Cv}.
\begin{Th}\label{te6.7}
We have
\begin{itemize}
\item[(A)]
\[
\begin{array}{l}
\displaystyle
\widetilde{\mathcal P}_i(\cdot;t)=
\frac{t-i+1}{t+1}\cdot\sum_{j=1}^{i+1} j\,\mathcal P_{j-1}\shuffle\widetilde{\mathcal P}_{i-j+1}(\cdot; t-j);
\end{array}
\]
\item[(B)]
\[
\widetilde{\mathcal P}_i(\cdot; t)=\sum_{j=1}^i  \left( \frac{(t-i+2)j}{i} -1\right) \mathcal P_j\shuffle\widetilde{\mathcal P}_{i-j}(\cdot; t-j).
\]
\end{itemize}
\end{Th}
(Recall that $\mathcal P_k(\alpha_1,\dots,\alpha_k):=\widetilde{\mathcal P}_k(\alpha_1,\dots,\alpha_k;k)$ and $\widetilde{\mathcal P}_0=1$.)
\begin{proof} (A) First, using the well-known recurrence relation for
the Bell polynomials,
\begin{equation}\label{eq6.14}
 B_{n,m}(x_1,\dots, x_{n-m+1}) = \sum_{j=1}^{n-m+1} \binom{n-1}{j-1} x_{j} B_{n-j,m-1}(x_1,\dots,x_{n-j-m+2}), 
 \end{equation}
we establish the following result for polynomials $\mathcal B_k$, see \eqref{genbell1}.

\begin{Proposition}\label{prop6.8}
For all $k\in\N$,
\begin{equation}\label{equ6.15}
\mathcal B_k(t_1,\dots, t_k, t)=\frac{t}{t+k}\cdot\sum_{j=1}^{k+1} j\,t_{j-1}\mathcal B_{k-j+1}(t_1,\dots, t_{k-j+1}, t-1);
\end{equation}
here and below $t_0:=1$.
\end{Proposition}
\begin{proof}
Applying \eqref{eq6.14} with $x_j:=j! t_{j-1}$, $1\le j\le n-m+1$, we get
\[
\begin{array}{l}
\displaystyle
B_{n,m}(1,2!t_1,\dots, (n-m+1)!t_{n-m})\medskip\\
\displaystyle = \sum_{j=1}^{n-m+1} \frac{(n-1)!}{(n-j)!(j-1)!} (j!t_{j-1}) B_{n-j,m-1} (1,2!t_1,\dots,(n-m+2-j)!t_{n-m+1-j}).
\end{array}
\]
This formula with $n=i+1$, $m=i-k+1$ and Proposition  \ref{prop6.3} imply for all integers $i\ge k$,
\[
\begin{array}{l}
\displaystyle \mathcal B_k(t_1,\dots, t_k, i-k+1)=\frac{(i-k+1)!}{(i+1)!}
B_{i+1,i-k+1}(1,2!t_1,\dots, (k+1)!t_k)\medskip\\
\displaystyle =\frac{(i-k+1)!}{(i+1)!}\cdot\sum_{j=1}^{k+1} \frac{i!}{(i+1-j)!}j\,t_{j-1}B_{i+1-j,i-k}(1,2!t_1,\dots , (k+2-j)!t_{k+1-j})\medskip\\
\displaystyle =\frac{(i-k+1)!}{(i+1)!}\cdot\sum_{j=1}^{k+1} \frac{i!}{(i+1-j)!}j\,t_{j-1}\cdot\frac{(i+1-j)!}{(i-k)!}\mathcal B_{k+1-j}(t_1,\dots, t_{k+1-j},i-k)\medskip\\
\displaystyle =\frac{i-k+1}{i+1}\sum_{j=1}^{k+1} j\,t_{j-1}\mathcal B_{k+1-j}(t_1,\dots, t_{k+1-j},i-k).
\end{array}
\]
This shows that \eqref{equ6.15} is valid for all
$t=i-k+1$ where $i\ge k$ is an integer number. Since both parts of the required identity comprise rational functions in $t$, the latter implies that \eqref{equ6.15} is valid for all $t\in\mathbb C$, as required.
\end{proof}
Now, identity (A) follows directly from Proposition \ref{prop6.8} and Theorem \ref{gpshuffle}.  Indeed, according to these results 
for each character $g\in G_{\mathbb C\langle\mathscr A\rangle}(\mathbb C)$,
\[
\begin{array}{l}
\displaystyle
g\left(\widetilde{\mathcal P}_i(\alpha_1,\dots,\alpha_i;t)\right)=g\left(\mathcal B_i(\mathcal P_1(\alpha_1),\dots, \mathcal P_{i}(\alpha_1,\dots,\alpha_i), t-i+1)\right)\medskip\\
\displaystyle
=\frac{t-i+1}{t+1}\cdot\sum_{j=1}^{i+1} j\,g_{j-1}\cdot\mathcal B_{i-j+1}(g_1,\dots , g_{i-j+1}, t-i))\medskip\\
\displaystyle =g\left(\frac{t-i+1}{t+1}\cdot\sum_{j=1}^{i+1} j\, \mathcal P_{j-1}(\alpha_1,\dots,\alpha_{j-1})\shuffle\widetilde{\mathcal P}_{i-j+1}(\alpha_1,\dots,\alpha_{i-j+1}; t-j)\right),
\end{array}
\]
where $g_k:=g\left(\mathcal P_k(\alpha_1,\dots,\alpha_j)\right)$,
$0\le k\le i$. 

This gives identity \eqref{equ6.15}.\medskip

\noindent (B) Using the recurrence relation for the Bell polynomials established in \cite{Cv},
\begin{equation}\label{eq6.15}
\begin{array}{l}
\displaystyle
B_{n,m}(x_1,\dots, x_{n-m+1})\medskip\\
\displaystyle =\frac{1}{x_1}\cdot\frac{1}{n-m}\sum_{j=1}^{n-m}\binom{n}{j}\left(m+1-\frac{n+1}{j+1}\right) x_{j+1}B_{n-j,m}(x_1,\dots, x_{n-j-m+1}),
\end{array}
\end{equation}
we prove the following result for polynomials $\mathcal B_k$.

\begin{Proposition}\label{prop6.9}
For all $k\in\N$
\begin{equation}\label{equ6.17}
\mathcal B_k(t_1,\dots, t_k, t)=\sum_{j=1}^k  \left( \frac{(t+1)\,j}{k} -1\right) t_j\mathcal B_{k-j}(t_1,\dots, t_{k-j}, t).
\end{equation}
\end{Proposition}

\begin{proof}
Applying Proposition \ref{prop6.3} and \eqref{eq6.15} with $x_j:=j! t_{j-1}$, $1\le j\le k+1$, 
$n=i+1$, $m=i-k+1$, $i\ge k$,  we get
\[
\begin{array}{l}
\displaystyle \mathcal B_k(t_1,\dots, t_k, i-k+1)=\frac{(i-k+1)!}{(i+1)!}
B_{i+1,i-k+1}(1,2!t_1,\dots, (k+1)!t_k)\medskip\\
\displaystyle =\frac{(i-k+1)!}{(i+1)!}\frac{1}{k}\sum_{j=1}^k\binom{i+1}{j}\left(i-k+2-\frac{i+2}{j+1}\right) x_{j+1} B_{i+1-j,i-k+1}(x_1,\dots,x_{k-j+1})\medskip\\
\displaystyle =\frac{1}{k}\sum_{j=1}^k\frac{(i+1-k)!}{(i+1-j)!}((i-k+2)(j+1)-(i+2)) t_j B_{i+1-j,i-k+1}(1,\dots,(k-j+1)!t_{k-j})\medskip\\
\displaystyle =\sum_{j=1}^k\frac{(i+1-k)!}{(i+1-j)!}\left( \frac{(i-k+2)j}{k} -1\right) t_j\cdot\frac{(i+1-j)!}{(i-k+1)!} \mathcal B_{k-j}(t_1,\dots, t_{k-j}, i-k+1)\medskip\\
\displaystyle =\sum_{j=1}^k \left( \frac{(i-k+2)j}{k} -1\right) t_j \mathcal B_{k-j}(t_1,\dots, t_{k-j}, i-k+1).
\end{array}
\]
Thus identity \eqref{equ6.17} is valid for all
$t=i-k+1$ where $i\ge k$ is an integer number. Since both parts of the required identity comprise polynomials in $t$, the latter implies that \eqref{equ6.17} is valid for all $t\in\mathbb C$, as required.
\end{proof}

As in the proof of (A), identity (B) follows directly from Proposition \ref{prop6.9} and Theorem \ref{gpshuffle}.  Indeed, according to these results 
for each character $g\in G_{\mathbb C\langle\mathscr A\rangle}(\mathbb C)$,
\[
\begin{array}{l}
\displaystyle
g\left(\widetilde{\mathcal P}_i(\alpha_1,\dots,\alpha_i;t)\right)=g\left(\mathcal B_i(\mathcal P_1(\alpha_1),\dots, \mathcal P_{i}(\alpha_1,\dots,\alpha_i), t-i+1)\right)\medskip\\
\displaystyle =\sum_{j=1}^i  \left( \frac{(t-i+2)\,j}{i} -1\right) g_j\mathcal B_{i-j}(g_1,\dots, g_{i-j}, t-i+1)\medskip\\
\displaystyle = g\left( \sum_{j=1}^i  \left( \frac{(t-i+2)j}{i} -1\right) \mathcal P_j(\alpha_1,\dots,\alpha_j)\shuffle\widetilde{\mathcal P}_{i-j}(\alpha_1,\dots, \alpha_{i-j}; t-j)\right),
\end{array}
\]
where $g_k:=g\left(\mathcal P_k(\alpha_1,\dots,\alpha_j)\right)$,
$1\le k\le i$.

This proves identity \eqref{equ6.17}.
\end{proof}
\subsection{Augmentation Homomorphism}
Recall that the augmentation homomorphism $\varepsilon_{\mathbb C}:\mathbb C\langle\mathbf X\rangle \rightarrow\mathbb C$ is defined on the generators by the formula $\varepsilon_\mathbb C(X_k)=1$, $k\in\mathbb Z_+$, $X_0:=I$. In this section we compute the values of $\varepsilon_\mathbb C$ on the generalized Devlin polynomials $\widetilde{\mathcal P}_i(X_1,\dots, X_i;t)$, $i\in\N$.

\begin{Proposition}\label{prop6.10} For all $i\in\N$, $t, x\in\mathbb C$,
\[
\begin{array}{r}
\displaystyle
\varepsilon_\mathbb C\bigl(\widetilde{\mathcal P}_i(xX_1,\dots, xX_i;t)\bigr):=\sum_{i_1+\cdots +i_k=i}(t-i_1+1)(t-i_1-i_2+1)\cdots (t-i+1)x^k\ \medskip\\
\displaystyle  =(xt+1)(x(t-1)+1)\cdots (x(t-i+2)+1) x(t-i+1).
\end{array}
\]
In particular,
\[
\varepsilon_\mathbb C\bigl(\mathcal P_i(X_1,\dots, X_i)\bigr):=\sum_{i_1+\cdots +i_k=i}(i-i_1+1)(i-i_1-i_2+1)\cdots 1=\frac{(i+1)!}{2}.
\]
\end{Proposition}
\begin{proof}
We give an operatorial proof of this result.
For an operator $A$ on $\mathbb C[z]$ we assume that $A^0:=I$. We have
\begin{equation}\label{equ6.18}
\begin{array}{lr}
\displaystyle
\left(I-\sum_{j=1}^\infty x DL^{j-1} t^j\right)^{-1}=\sum_{i=0}^\infty\left(\sum_{j=1}^\infty x DL^{j-1} t^j\right)^i\medskip\\
\displaystyle \qquad\qquad\qquad\qquad \qquad\ \ =I+\sum_{i=1}^\infty\left(\sum_{i_1+\cdots +i_k=i}x^k DL^{i_1-1}\cdots DL^{i_k-1} \right) t^i.
\end{array}
\end{equation}

On the other hand,
\begin{equation}\label{equ6.19}
\begin{array}{r}
\displaystyle \left(I-\sum_{j=1}^\infty x DL^{j-1} t^j\right)^{-1}=
\bigl(I-xtD (I-tL)^{-1}\bigr)^{-1}=(I-tL)(I-tL-xtD)^{-1}\medskip\\
\displaystyle =(I-tL)\sum_{i=0}^\infty (L+xD)^i\, t^i=I+\sum_{i=1}^\infty \bigl((L+xD)^i- L(L+xD)^{i-1}\bigr)t^i\medskip\\
\displaystyle =I+\sum_{i=1}^\infty xD (L+xD)^{i-1} t^i.
\end{array}
\end{equation}
Equating the coefficients of $t^i$ on the right-hand sides of \eqref{equ6.18} and \eqref{equ6.19} we get
\begin{equation}\label{equ6.20}
\sum_{i_1+\cdots +i_k=i}x^k DL^{i_1-1}\cdots DL^{i_k-1} =xD (L+xD)^{i-1}.
\end{equation}
Now, applying both sides of \eqref{equ6.20} to $z^m$ with $m\ge i$ we obtain
\[
\sum_{i_1+\cdots +i_k=i}p_{i_1,\dots, i_k}(m)x^k z^{m-i}=(xm+1)(x(m-1)+1)\cdots (x(m-i+2)+1) x(m-i+1)z^{m-i}.
\]

Since in both parts of the required identity are polynomials and the latter shows that this identity is valid for all integers $t\ge i$ and $x\in\mathbb C$, it is valid for all $t\in\mathbb C$ as well. This completes the proof of the proposition.
\end{proof}
\begin{C}\label{cor6.11}
Given $i, k\in\N$, $k\le i$,
\[
\begin{array}{l}
\displaystyle
S_{i,k}(t):= \sum_{i_1+\cdots +i_k=i}(t-i_1+1)(t-i_1-i_2+1)\cdots (t-i+1)\medskip\\
\displaystyle \quad\qquad=(t-i+1)\cdot\sum_{l=0}^{k-1}  s(i-1,l+i-k)\,\binom{l+i-k}{l}\,t^{l},
\end{array}
\]
where $s(\cdot,\cdot)$ are Stirling numbers of the first kind.\footnote{In contrast with our previous notation, in the above sum $i$ and $k$ are fixed numbers.}
\end{C}
Note that $S_{i,k}(i-2)=-|s(i-1,i-k)|$, $i\ge 2$;  $S_{i,k}(0)=(1-i)\cdot s(i-1,i-k)$.
\begin{proof} For $x\ne 0$ we have
\[
\begin{array}{l}
\displaystyle \sum_{k=1}^i S_{i,k}(t)\,x^k=(xt+1)(x(t-1)+1)\cdots (x(t-i+2)+1)\, x(t-i+1)\\
\\
\displaystyle
=x^i \left(t+\frac{1}{x}\right)\left(t+\frac 1x-1\right)\cdots \left(t+\frac 1x-(i-2)\right) (t-i+1)\\
\\
\displaystyle =\sum_{l=0}^{i-1}s(i-1,l)\left(t+\frac 1x\right)^l x^{i} (t-i+1)=\sum_{l=0}^{i-1}s(i-1,l)\sum_{m=0}^l \binom{l}{m}t^{l-m} x^{i-m}(t-i+1)\\
\\
\displaystyle =(t-i+1)\cdot \sum_{k=1}^i\left(\sum_{l=0}^{k-1}  s(i-1,l+i-k)\,\binom{l+i-k}{l}\,t^{l}\right) x^k.
\end{array}
\]
\end{proof}
\begin{E}
{\rm Consider a separable equation
\begin{equation}\label{eq6.21}
\frac{dv}{dx}=\sum_{i=1}^\infty v^{i+1},\qquad  x\in I_T:=[0,T].
\end{equation}
For all sufficiently small  initial values $v(0)=r\in\mathbb C$
its solution $v(\cdot;r)$ satisfies \penalty-10000 $\sup_{x\in I_T}|v(x,r)|<1$ and is given by the formula, cf. \eqref{e4} and Corollary \ref{cor6.11} with $t=i$, 
\begin{equation}\label{eq6.22}
\begin{array}{l}
\displaystyle
v(x;r)=r+\sum_{i=1}^{\infty}\left(\sum_{i_{1}+\cdots +i_{k}=i}p_{i_{1},\dots, i_{k}}(i)\cdot \int\cdots\int_{0\leq s_{1}\leq\cdots\leq s_{k}\leq x}ds_{k}\cdots ds_{1}\right)r^{i+1}\medskip\\
\displaystyle \qquad\quad =r+\sum_{i=1}^{\infty}\left(\sum_{k=1}^i\frac{S_{i,k}(i)}{k!}\, x^k \right)r^{i+1},\qquad x\in I_T.
\end{array}
\end{equation}
On the other hand, for such $r$, 
\[
\sum_{i=1}^\infty v(x;r)^{i+1}=\frac{v^{2}(x;r)}{1-v(x,r)},\qquad x\in I_T,
\]
so that \eqref{eq6.21} is equivalent to equation
\begin{equation}\label{eq6.23}
\frac{dv}{dx}=\frac{v^2}{1-v},\qquad  x\in I_T:=[0,T].
\end{equation}
For real initial values $r\ne 0$ it can be solved explicitly producing solutions
\begin{equation}\label{eq6.24}
v(x;r)=\left\{
\begin{array}{lll}
\displaystyle -\frac{1}{W_0\left(e^x\cdot W_0^{-1}\bigl(-\frac 1r\bigr)\right)}&{\rm if}&-\infty< r<0\medskip\\
\displaystyle -\frac{1}{W_{-1}\left(e^x\cdot W_{-1}^{-1}\bigl(-\frac 1r\bigr)\right)}&{\rm if}&\displaystyle 0<r\le -\frac{1}{W_{-1}(-e^{-T-1})}
\end{array}
\right.
\end{equation}
such that
\[
\lim_{r\rightarrow 0^{-}}\sup_{x\in I_T}|v(x,r)|= \lim_{r\rightarrow 0^{+}}\sup_{x\in I_T}|v(x,r)|=0.
\]
Here $W_0^{-1}(s)=W_{-1}^{-1}(s)=s e^s$, $s\in\mathbb R$, and $W_0$, $W_{-1}$ are real branches of the Lambert function, see, e.g., \cite{CGHJK} for their properties.

Due to the uniqueness of solutions of initial value problems of \eqref{eq6.21} and \eqref{eq6.23}, for all sufficiently small real $r\ne 0$, functions in \eqref{eq6.22} and \eqref{eq6.24} coincide. Thus \eqref{eq6.24} determines a real analytic function on $I_T\times \left(-\infty, -\frac{1}{W_{-1}(-e^{-T-1})}\right)$ satisfying \eqref{eq6.23} with the Taylor series expansion about $(0,0)$ given by \eqref{eq6.22}.

Finally, note that from Theorem \ref{gpshuffle} we obtain for all $i\in\N$, $t\in\mathbb C$,
\begin{equation}\label{eq6.25}
S_{i}(x,t)=\mathcal B_i(S_1(x,1),\dots, S_i(x,i), t-i+1),
\end{equation}
where
\[
S_i(x,t):=\sum_{k=1}^i\frac{S_{i,k}(t)}{k!}\, x^k.
\]
}
\end{E}

\end{document}